\documentclass[oneside,11pt]{amsart}
\usepackage[english]{babel}
\usepackage{mdwlist}
\usepackage{upgreek}
\usepackage{bbm} 
\usepackage{euscript}
\usepackage{mathrsfs}
\usepackage{amsfonts}
\usepackage{amsthm}
\usepackage{amssymb}
\usepackage{amsmath}
\usepackage{fancybox}
\usepackage{graphicx}
\usepackage{color}
\usepackage{lineno}
\usepackage[abbrev]{amsrefs}
\usepackage{rotating}
\usepackage{graphicx,graphics,epsf}
\usepackage{mathcomp} 
\usepackage[utf8x]{inputenx} %
\usepackage[LGR, T1]{fontenc}
\usepackage{textcomp} 
\usepackage{lmodern}

\newtheorem{thm}{Theorem} 
\newtheorem{thm*}{Theorem*}
\newtheorem{lemma}[thm]{Lemma} 
\newtheorem{theorem}[thm]{Theorem}

\newtheorem{proposition}[thm]{Proposition}

\newtheorem{remark}[thm]{Remark} 
\newtheorem{claim}[thm]{Claim}
 
\newtheorem{definition}[thm]{Definition}

\DeclareMathOperator{\mdl}{\rm mod}

\def\LL{{\mathscr L}}

\def\od#1{{\overline{d}(#1)}}

\def\LL{{\mathcal L}}

\def\tred#1{{\textcolor{black}{#1}}}

\def\floor#1{{\lfloor{#1}\rfloor}}
\def\bigfloor#1{{\bigl\lfloor{#1}\bigr\rfloor}}

\def\ceil#1{{\lceil{#1}\rceil}}

\def\ignore#1{}

\def\rot#1#2{{\text{\rm rot}}_{#1}(#2)}
\def\Bur{{\overline{C}_6}}
\def\Rot#1{{\text{\rm Rot}{(#1)}}}
\def\MRot#1{{\text{\rm Rot}_{M}{(#1)}}}
\def\Key#1{{\Upphi(#1)}}
\def\PKey#1{{\Upphi^{1}(#1)}}
\def\cc{{\mathcal C}}

\newcommand{\cru}{\operatorname{cr}}

\def\ucr#1#2{{\text{\rm cr}}_{#1}(#2)}

\def\floor#1{{\lfloor{#1}\rfloor}}
\def\bigfloor#1{{\biggl\lfloor{#1}\biggr\rfloor}}

\def\ceil#1{{\lceil{#1}\rceil}}

\title[The optimal drawings of $K_{5,n}$]{
  The optimal drawings of $K_{5,n}$
        }

\author{C\'esar Hern\'andez-V\'elez}
\address{Instituto de F\'{\i}sica, UASLP. San Luis Potos\'{\i}, M\'exico.}
\email{cesar@ifisica.uaslp.mx}
\author{Carolina Medina}
\address{Instituto de F\'{\i}sica, UASLP. San Luis Potos\'{\i}, M\'exico.}
\email{cmedina@ifisica.uaslp.mx}
\author{Gelasio Salazar}
\address{Instituto de F\'{\i}sica, UASLP. San Luis Potos\'{\i}, M\'exico.}
\email{gsalazar@ifisica.uaslp.mx}
\thanks{The third author was supported by CONACYT grant 106432.}

\date{\today}



\begin{document}

\linenumbers

\reversemarginpar

\begin{abstract}
Zarankiewicz's Conjecture (ZC) states that the crossing number
$\cru(K_{m,n})$ equals $Z(m,n):=\floor{\frac{m}{2}}
\floor{\frac{m-1}{2}} \floor{\frac{n}{2}} \floor{\frac{n-1}{2}}$. Since
Kleitman's verification of ZC for $K_{5,n}$ (from which ZC for
$K_{6,n}$ easily follows), very little progress has been made around
ZC; the most notable exceptions involve computer-aided results. With
the aim of gaining a more profound understanding of this notoriously
difficult conjecture, we investigate the {\em optimal} (that is,
crossing-minimal) drawings of
$K_{5,n}$. The widely known natural drawings of $K_{m,n}$ (the
so-called {\em Zarankiewicz drawings}) with
$Z(m,n)$ crossings contain {\em antipodal} vertices, that is, pairs of
degree-$m$ vertices such that their induced drawing of $K_{m,2}$ has
no crossings. Antipodal vertices also play a major role in Kleitman's
inductive proof that $\cru(K_{5,n}) = Z(5,n)$.  We explore in depth
the role of antipodal vertices in optimal drawings of $K_{5,n}$, for
$n$ even. We prove that if
{$n \equiv 2$ (mod $4$)}, then every optimal drawing
of $K_{5,n}$ has antipodal vertices. We also exhibit a two-parameter
family of optimal drawings $D_{r,s}$ of $K_{5,4(r+s)}$ (for $r,s\ge 0$), with
no antipodal vertices, and show that if $n\equiv 0$ (mod $4$), then
every optimal drawing of $K_{5,n}$ without
antipodal vertices is (vertex rotation) isomorphic to $D_{r,s}$ for some
integers $r,s$.
As a
corollary, we show that if $n$ is even, then
every optimal drawing of
$K_{5,n}$ is the superimposition of Zarankiewicz drawings with a
drawing isomorphic to $D_{r,s}$ for some nonnegative integers $r,s$.
\end{abstract}

\maketitle

\section{Introduction. }\label{sec:intro}

We recall that the {\em crossing number} $\cru(G)$ of a graph $G$ is
the minimum number of pairwise crossings of edges in a drawing of $G$
in the plane. A drawing of a graph is {\em good} if no adjacent edges
cross, and no two edges cross each other more than once. It is trivial
to show that every {\em optimal} (that is, crossing-minimal)
drawing of a graph is good.


One of the most tantalizingly open crossing number questions was
raised by Tur\'an in 1944: what is the crossing number
$\cru(K_{m,n})$ of the
complete bipartite graph $K_{m,n}$? Zarankiewicz~\cite{zar2} described how
to draw $K_{m,n}$ with exactly $Z(m,n)$ crossings, where 
\[
Z(m,n):=\bigfloor{\frac{m}{2}} \bigfloor{\frac{m-1}{2}}
\bigfloor{\frac{n}{2}} \bigfloor{\frac{n-1}{2}}.
\]

\begin{figure}[h]
\begin{center}
\scalebox{0.7}{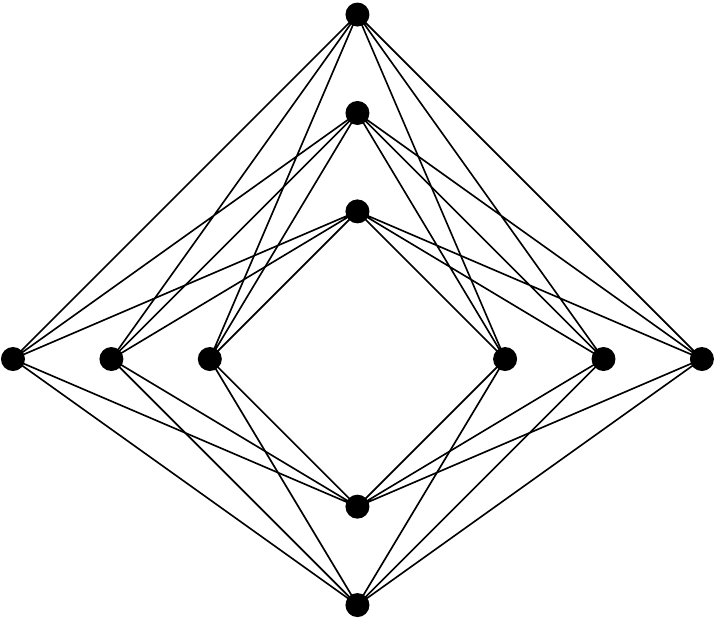}
\caption{Drawing of $K_{5,6}$ with $Z(5,6)=24$ crossings.}
\label{fig:kmnzar}
\end{center}
\end{figure}

Zarankiewicz's construction is shown in Figure~\ref{fig:kmnzar}
for the case $m=5, n=6$. It is straightforward to generalize this
drawing to a drawing of $K_{m,n}$ with $Z(m,n)$ crossings, for all
positive integers $m$ and $n$, and so $\cru(K_{m,n})\le
Z(m,n)$. The drawings thus obtained are the {\em Zarankiewicz
  drawings} of $K_{m,n}$.

In~\cite{zar2},
Zarankiewicz claimed to have proved that $\cru(K_{m,n})=Z(m,n)$ for
all positive integers $m,n$. However, Kainen and Ringel independently found a flaw in
Zarankiewicz's argument (see~\cite{decline}), and the statement
$\cru(K_{m,n}) = Z(m,n)$ has become known as {\em Zarankiewicz's
    Conjecture}. 

Very little of substance is known about $\cru(K_{m,n})$. An elegant
argument using $\cru(K_{3,3})=1$ plus purely combinatorial arguments
(namely, Tur\'an's theorem on the maximum
number of edges in a triangle-free graph) shows that $\cru(K_{3,n}) =
Z(3,n)$. An easy counting argument shows that  
$\cru(K_{2s-1,n}) = Z(2s-1,n)$ (for any $s \ge 1$) implies 
that $\cru(K_{2s,n}) = Z(2s,n)$. Thus it follows that $\cru(K_{4,n}) =
Z(4,n)$. Kleitman~\cite{kleitman} proved that
$\cru(K_{5,n})=Z(5,n)$. By our previous remark, this
implies that $\cru(K_{6,n}) = Z(6,n)$. 

After Kleitman's theorem, most progress around
Zarankiewicz's Conjecture consists of computer-aided
results. Woodall~\cite{woodall} verified Zarankiewicz's Conjecture for $K_{7,7}$ and
$K_{7,9}$. De Klerk et al.~\cite{dmp}~used semidefinite programming techniques to
show that $\lim_{n\to\infty} \cru(K_{7,n})/Z(7,n) \ge 0.968$. Also
using semidefinite programming and deeper algebraic techniques, De Klerk
et al.~\cite{dps} proved that
$\lim_{n\to\infty} \cru(K_{9,n})/Z(9,n) \ge 0.966$. In a related result, De
Klerk and Pasechnik~\cite{ds} recently showed that the $2$-page crossing number
$\nu_2(K_{7,n})$ of $K_{7,n}$ satisfies $\lim_{n\to\infty}
\cru(K_{7,n})/Z(7,n) =1$. 

We finally mention that recently Christian et al.~\cite{chri} proved that
deciding Zarankiewicz's Conjecture is a finite problem for each fixed $m$.




To give a brief description of our results, 
let us color the $5$ degree-$n$ vertices of
$K_{5,n}$ {\em black}, and color the $n$ degree-$5$ vertices {\em
  white}.  
Two white vertices are {\em antipodal} in a drawing $D$ of
$K_{5,n}$ if the drawing of the $K_{5,2}$ they induce has no crossings. A drawing is
{\em antipodal-free} if it has no antipodal vertices.

Antipodal pairs are evident in Zarankiewicz's drawings (moreover, the
set of white vertices can be decomposed into two classes, such that
any two white vertices in distinct classes are antipodal). Antipodal
pairs are also crucial in the inductive step of Kleitman's proof,
which does not 
concern itself with the different ways (if more than one) to achieve
$Z(5,n)$ crossings with a drawing of $K_{5,n}$.

Given their preeminence in Zarankiewicz's Conjecture, we set out to
investigate the role of antipodal pairs in the optimal drawings of
$K_{5,n}$.  Our main result (Theorem~\ref{thm:main1}) characterizes
optimal drawings of $K_{5,n}$, for even $n$, as follows.  First, if
{$n \equiv 2$ (mod $4$)}, then all optimal drawings of $K_{5,n}$ have
antipodal pairs. Second, if { $n \equiv 0$ (mod $4$)}, then every
antipodal-free optimal drawing of $K_{5,n}$ is isomorphic (we review
vertex rotation isomorphism in Section~\ref{sec:rots}) to a drawing in
a two-parameter family $D_{r,s}$ of drawings we have fully
characterized. As a consequence of these facts, we show
(Theorem~\ref{thm:main2}) that if $n$ is even, then every optimal
drawing of $K_{5,n}$ can be obtained by starting with $D_{r,s}$, for
some nonnegative (possibly zero) integers $r$ and $s$, and then
superimposing Zarankiewicz drawings.

The rest of this paper is organized as follows.  In
Section~\ref{sec:rots} we review the concept of vertex rotation, which
is central to the criterion to decide when two drawings are
isomorphic. In Section~\ref{sec:spdr} we describe the two-parameter
family of optimal, antipodal-free drawings $D_{r,s}$ (for integers
$r,s\ge 0$) of $K_{5,4(r+s)}$. In
Section~\ref{sec:main1} we state our main results.
Theorem~\ref{thm:main1} claims that (i) if {$n\equiv 2$ (mod $4$)}, then
every optimal drawing of $K_{5,n}$ has antipodal vertices; and that
(ii) if {$n\equiv 0$ (mod $4$)}, then every antipodal-free optimal
drawing of $K_{5,n}$ is isomorphic to $D_{r,s}$ for some integers
$r,s$ such that $4(r+s)=n$. In Theorem~\ref{thm:main2}
we state the decomposition of optimal drawings of $K_{5,n}$, along the
lines of the previous paragraph.  The proof of Theorem~\ref{thm:main2}
is also given in this section; the rest of the paper is devoted to the
proof of Theorem~\ref{thm:main1}.  In Section~\ref{sec:clean} we
introduce the concept of a {\em clean} drawing. Loosely speaking, a drawing
is clean if its white vertices can be naturally partitioned into
{\em bags}, so that vertices in the same bag have the same (crossing
number wise) properties. In Section~\ref{sec:keyscores} we introduce
{\em keys}, which are labelled graphs that capture the essential (crossing
number wise) information of a clean drawing. This abstraction (and the
related concept of {\em core}) will prove to be extremely useful for the
proof of Theorem~\ref{thm:main1}. In Section~\ref{sec:someprco} we
investigate which labelled graphs can be the key of a
relevant (clean, optimal, antipodal-free) drawing.  Cores are certain more
manageable subgraphs of keys, that retain all the (crossing number
wise) useful information of a key. 
We devote Sections~\ref{sec:someprrk1},
\ref{sec:someprrk2}, \ref{sec:someprrk3}, and 
\ref{sec:someprrk4} to the task of completely characterizing which
graphs can be the core of an antipodal-free optimal drawing. The
information in these sections is then put together in
Section~\ref{sec:thecores}, where we show that the core of every
optimal drawing is isomorphic either to the $4$-cycle or to the graph
$\Bur$ obtained by adding to the $6$-cycle a diametral edge.
The proof of Theorem~\ref{thm:main1}, 
given in Section~\ref{sec:proofmain}, is an easy consequence of this
full characterization of cores.

\section{Rotations and isomorphic drawings.}\label{sec:rots}

To help comprehension, throughout this paper we color the $5$ {degree-$n$} vertices in
$K_{5,n}$ {\em black}, and the $n$ {degree-$5$} vertices {\em white}. We label the
black vertices $0,1,2,3,4$. Unless otherwise stated, we label the
white vertices $a_0, a_1, \ldots, a_{n-1}$. {We adopt the notation}
{${[n]}:=\{0,1,\ldots,n-1\}$.}

Given vertices $a_i, a_j$ with $i,j\in {{[n]}}$, we let
$S(a_i)$ denote the {\em star} centered at $a_i$, that is, the
subgraph (isomorphic to $K_{5,1}$) induced by $a_i$ and the vertices
$0,1,2,3,4$.  If $D$ is a drawing of $K_{5,n}$, we let
$\ucr{D}{a_i,a_j}$ denote the number of crossings in $D$ that involve
an edge of $S(a_i)$ and an edge of $S(a_j)$, and we let
$\ucr{D}{a_i}:=\sum_{k\in{{[n]}},k\neq i}
\ucr{D}{a_i,a_k}$.
Formalizing the definition from Section~\ref{sec:intro}, $a_i$ and $a_j$ are {\em antipodal} ({\em in} $D$) if
$\ucr{D}{a_i,a_j}=0$.

The {\em rotation} $\rot{D}{a_i}$ of a white
vertex $a_i$ in a drawing $D$ is the cyclic permutation that records
the (cyclic) counterclockwise order in which the
edges leave $a_i$. 
We use the notation $01234$ for permutations, and
$(01234)$ for cyclic permutations. For instance, the rotation
$\rot{D}{a_3}$ of the vertex $a_3$ in the drawing $D$ in
Figure~\ref{fig:rot01} is $(02431)$: 
following a counterclockwise
order, if we start with the edge leaving from $a_3$ to $0$, then we
encounter the edge leaving to $2$, then the edge leaving to $4$, then
the edge leaving to $3$, and then the edge leaving to $1$. We emphasize
that a rotation is a cyclic permutation; that is,
$(02431),(24310),(43102),(31024)$, and $(10243)$ denote (are) the same
rotation.
We let $\Pi$ denote the set of all cyclic permutations of
{$0,1,2,3,4$}. Clearly, $|\Pi|= 5!/5=4! = 24$. 
The {\em rotation} $\rot{D}{i}$ of a black vertex $i$ is defined
analogously: for each {$i\in {{[5]}}$},  $\rot{D}{i}$ is a cyclic permutation of $a_0,
a_1, \ldots, a_{n-1}$.

The {\em rotation multiset $\MRot{D}$ of $D$} is the multiset (that is,
repetitions are allowed) containing the $n$ rotations $\rot{D}{a_i}$,
for $i=0,1,\ldots,n-1$. The {\em rotation set} $\Rot{D}$ of $D$ is the
underlying set (that is, no repetitions allowed) of $\MRot{D}$.
Thus, in the example of
Figure~\ref{fig:rot01}, {$\MRot{D} = [(04321),$ $ (04321), (01234), (02431)
]$} {(we use square brackets for multisets)}, and 
{$\Rot{D} = 
\{ 
(04321),(01234),$ $(02431)
\}$.} 

	\begin{figure}[h!]
	\begin{center}
		\scalebox{0.7}{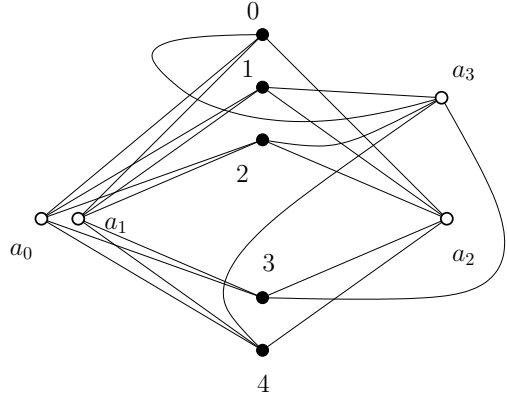}
		 \caption{\small{A drawing $D$ of $K_{5,4}$ with
                     $\rot{D}{a_0}=\rot{D}{a_1}=(04321),
                     \rot{D}{a_2}=(01234)$, and
                     $\rot{D}{a_3}=(02431)$. Thus the pair $a_0,a_2$
                     (as well as the pair $a_1, a_2$) is antipodal.}}
		 \label{fig:rot01}
	\end{center}
	\end{figure}

Two multisets $M,M'$ of rotations are {\em equivalent} (we write
$M\cong M'$) if one of them can be obtained from the other by
a relabelling (formally, a self-bijection) of $0,1,2,3,4$. 
Two drawings $D,D'$ of $K_{5,n}$ are {\em isomorphic} if
{$\MRot{D}\cong\MRot{D'}$}.  Loosely speaking, two drawings $D,D'$ of
$K_{5,n}$ are isomorphic if $0,1,2,3,4$ and $a_0, a_1, \ldots,
a_{n-1}$ can be relabelled (say in $D'$), if necessary, so that
$\rot{D}{a_i} =\rot{D'}{a_i}$ for every $i\in [n]$.

Our ultimate interest lies in optimal drawings (of
$K_{5,n}$). It is not difficult to see (we will prove this later)
that  if $D$ is an optimal drawing and $a_i,a_j,a_k,a_\ell$ are
vertices such that 
$\rot{D}{a_i} = \rot{D}{a_j}$ and $\rot{D}{a_k} = \rot{D}{a_\ell}$,
then $\ucr{D}{a_i,a_k} = \ucr{D}{a_j,a_\ell}$. 
Thus an optimal drawing
of $K_{5,n}$ is adequately described by choosing a representative
vertex of each rotation, and giving the information of how many
vertices there are for each rotation. This supports the pertinence of
focusing on the rotations as the criteria for isomorphism.


\section{An antipodal-free drawing of $K_{5,4(r+s)}$}\label{sec:spdr}

In this section we describe an antipodal-free drawing $D_{r,s}$ of $K_{5,4(r+s)}$,
for each pair $r,s$ of nonnegative integers.

	\begin{figure}[ht!]
	\begin{center}
		\scalebox{0.5}{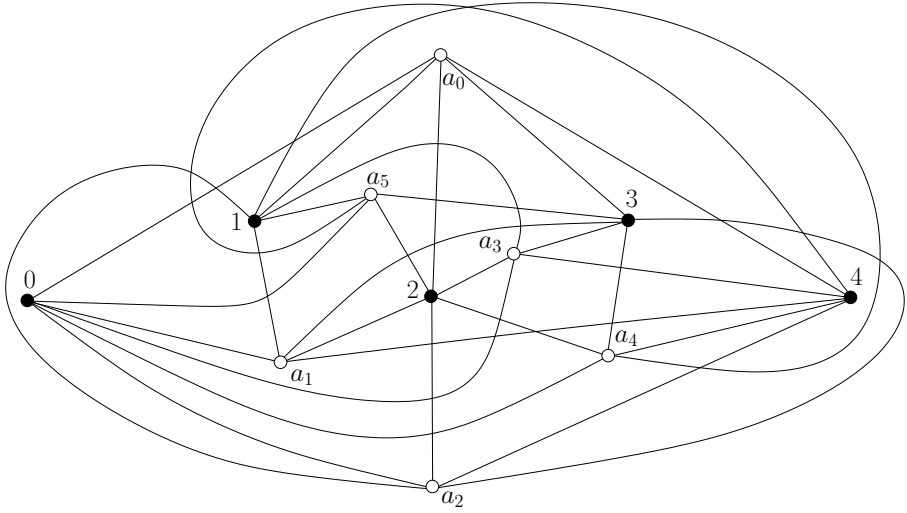}
		 \caption{\small{This antipodal-free drawing $D^*$
                     of $K_{5,6}$ is the base of the construction of the
                     optimal antipodal-free drawing $D_{r,s}$ of $K_{5,4(r+s)}$ for all $r, s$. It is easily verified that 
		$\rot{D^*}{a_0}=(01234)$,
		$\rot{D^*}{a_1}=(04231)$,
		$\rot{D^*}{a_2}=(01342)$,
		$\rot{D^*}{a_3}=(04312)$,
		$\rot{D^*}{a_4}=(01432)$,
		$\rot{D^*}{a_5}=(02314)$.
 }}
		 \label{fig:k5601}
	\end{center}
	\end{figure}

The construction is based on the drawing $D^*$ of $K_{5,6}$ in
Figure~\ref{fig:k5601}. As shown, the rotations in $D^*$ of the white
vertices are 
$\rot{D^*}{a_0}$ $=(01234)$,
$\rot{D^*}{a_1}=(04231)$,
$\rot{D^*}{a_2}=(01342)$,
$\rot{D^*}{a_3}$ $=(04312)$,
$\rot{D^*}{a_4}=(01432)$,
$\rot{D^*}{a_5}=(02314)$.

It is immediately checked that $D^*$ is antipodal-free. Note that
$D^*$ itself is not optimal, as it has $25=Z(5,6)+1$ crossings.

Suppose first that both $r$ and $s$ are positive.
To obtain $D_{r,s}$, we add $4(r+s)-6$ white vertices to $D^*$. Now $r-1$
of these vertices are drawn very close to $a_1$, and  $r-1$ are
drawn very close to $a_2$;  $s-1$ vertices are drawn very close to
$a_4$, and $s-1$ are drawn very close to $a_5$; finally, $r+s-1$ vertices are drawn
very close to $a_0$, and  $r+s-1$  are drawn very close to $a_3$. It is
intuitively clear what is meant by having $a_i$ drawn ``very close''
to $a_j$. Formally, we require that: (i)
$a_i$ and $a_j$ have the same rotation; (ii) $\ucr{D_{r,s}}{a_i,a_j}=4$; and (iii) for any other vertex $a_k$,
$\ucr{D_{r,s}}{a_i,a_k}=\ucr{D_{r,s}}{a_j,a_k}$. These properties are
easily satisfied by having the added vertex $a_i$ drawn sufficiently close
to $a_j$, so that the edges incident with $a_i$ follow very closely
the edges incident with $a_j$. 

If one of $r$ or $s$ is $0$, then we make the obvious adjustments. 
That is, (i) if  $r=0$, then we remove $a_1$ and $a_2$, and for each $i=0,3,4,5$,
we draw $s-1$ new vertices very close to $a_i$; and (ii) if $s=0$, then we
remove $a_4$ and $a_5$, and for each $i=0,1,2,3$, we draw $r-1$ new
vertices very close to $a_i$. (In the extreme case $r=s=0$, we remove
all the white vertices from $D^*$, and are left with an obviously
optimal drawing of $K_{5,0}$). 

For each $i=0,1,2,3,4,5$, the {\em bag} $[a_i]$ of $a_i$ is the set that
consists of the vertices drawn very close to $a_i$, plus
$a_i$ itself.

Note that each of  $[a_0]$ and $[a_3]$ has $r+s$
vertices, each of $[a_1]$ and $[a_2]$ has $r$ vertices, and each of
$[a_4]$ and $[a_5]$ has $s$ vertices.

An illustration of the construction for $r=2$ and $s=1$ is given in
Figure~\ref{fig:rho}, where the gray vertices are the ones added to $D^*$.

	\begin{figure}[ht!]
	\begin{center}
		\scalebox{0.5}{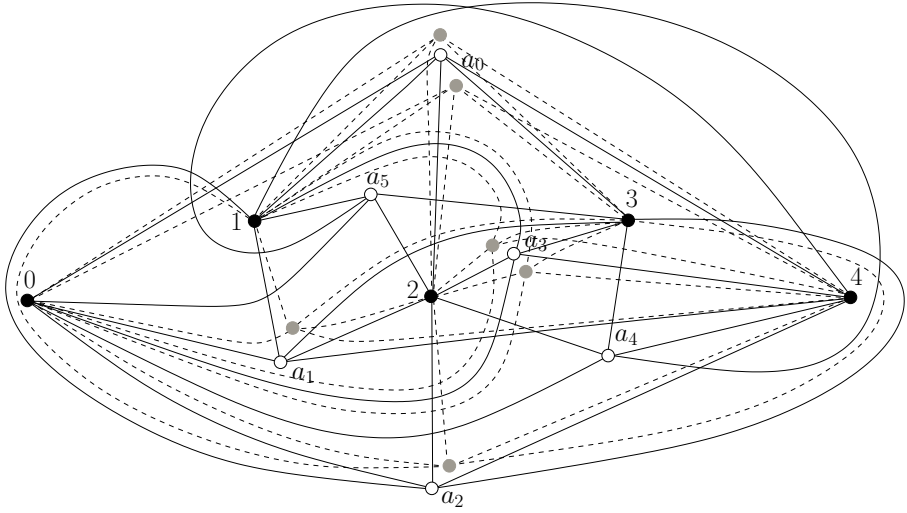}
		 \caption{\small{The antipodal-free drawing
                     $D_{2,1}$. To obtain this optimal drawing of
                     $K_{5,12}=K_{5,4(2+1)}$, we start with the drawing in
                    Figure~\ref{fig:k5601} and add two vertices very
                    close to $a_0$, two vertices very close to $a_3$,
                    one vertex very close to $a_1$, and one vertex
                    very close to $a_2$. Since $s-1=0$, no vertices are
                    added very close to either 
                    $a_4$ or $a_5$. The added vertices are colored gray in this drawing.}}
		 \label{fig:rho}
	\end{center}
	\end{figure} 

\medskip
\noindent{\bf Claim. }{\em
For every pair $r,s$ of nonnegative integers,  $D_{r,s}$ is an antipodal-free optimal drawing of
$K_{5,4(r+s)}$.}
\medskip

\begin{proof}
First we note that since $D^*$ is antipodal-free, it follows immediately that $D_{r,s}$
is also antipodal-free. Thus we only need to prove optimality. 

An elementary calculation gives the number of crossings in
$D_{r,s}$. For instance, take a vertex $u$ in $[a_0]$. Now $\cru_{D_{r,s}}(u,v)$ equals
(i) $4$ if $v \in [a_0], v\neq u$;
(ii) $1$ if $v\in [a_1]$;
(iii) $2$ if $v\in [a_2]$;
(iv) $1$ if $v\in [a_3]$;
(v) $1$ if $v\in [a_4]$; and
(vi) $2$ if $v\in [a_5]$.
Since $|[a_0]|= r+s$, 
$|[a_1]|= r$, 
$|[a_2]|= r$, 
$|[a_3]|= r+s$, 
$|[a_4]|= s$, and
$|[a_5]|= s$, it follows that $\ucr{D_{r,s}}{u}= 4(r+s-1) + r + 2r +
(r+s) + s + 2s = 
4(2r+2s-1)$. 

A totally analogous argument shows that, actually, 
$\ucr{D_{r,s}}{w}=4(2r+2s-1)$ for {\em every} white vertex $w$. 
Since there are $4(r+s)$ white vertices in total,  it
follows that $\cru(D_{r,s}) = (1/2)\bigl(4(r+s)\bigr)\bigl(4(2r+2s-1)\bigr) = \bigl(4(r+s)\bigr)\bigl(4(r+s)-2\bigr) =
Z(5,4(r+s))$. 
\end{proof}

\section{Main results: the optimal drawings of $K_{5,n}$, for $n$ even.}\label{sec:main1}

We now state our main results.


\begin{theorem}\label{thm:main1}
Let $n$ be a positive even integer. 
\begin{enumerate}
\item\label{it:ma1} If {$n\equiv 2$ \text{\rm (mod} $4$\text{\rm )}}, then all optimal drawings of $K_{5,n}$ have
antipodal vertices. 
\item\label{it:ma2} 
If {$n\equiv 0$ \text{\rm (mod} $4$\text{\rm )}}, then every
antipodal-free optimal drawing of $K_{5,n}$ is isomorphic to
$D_{r,s}$ (described in Section~\ref{sec:spdr}) for some
integers $r,s$ such that $4(r+s)=n$.
\end{enumerate}
\end{theorem}
	\begin{figure}[h!]
	\begin{center}
		\scalebox{0.5}{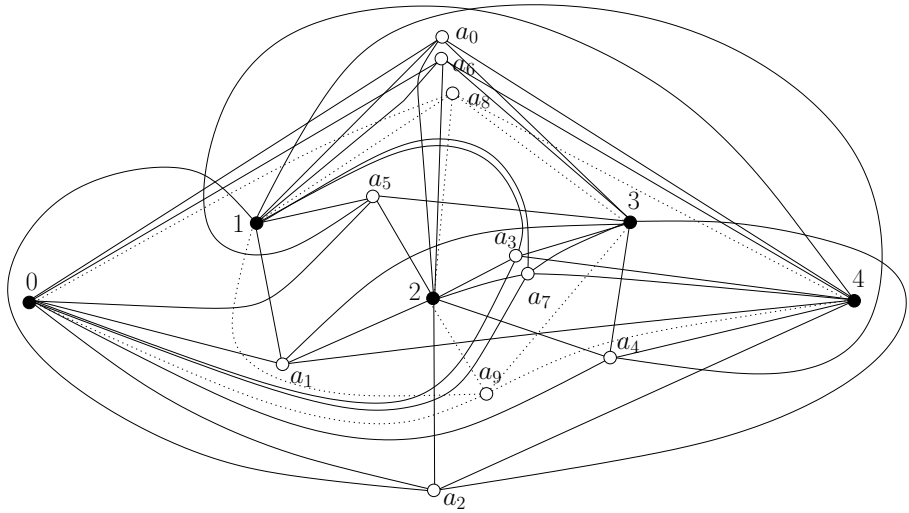}
                \caption{\small{An optimal drawing of $K_{5,10}$ that
                    is neither a Zarankiewicz drawing nor the
                    superimposition of Zarankiewicz drawings.  As
                    predicted by Theorem~\ref{thm:main2}, this is the
                    superimposition of a Zarankiewicz drawing (the $K_{5,2}$ induced by $a_8, a_9$ and the five
                    black vertices) plus a drawing $D_{r,s}$ (namely
                    with $r=s=1$).}}
		 \label{fig:anti}
	\end{center}
	\end{figure}

Before moving on to the proof of Theorem~\ref{thm:main1} (the rest of the
paper is devoted to this proof), we will show that it implies a decomposition of all the
optimal drawings of $K_{5,n}$, for $n$ even.

In Section~\ref{sec:intro} we defined, somewhat informally, a
Zarankiewicz drawing. Let us now formally define these drawings using
rotations (we focus on $K_{5,n}$, although the definition is obviously
extended to $K_{m,n}$ for any $m$). 
For a nonnegative integer $n$, a drawing $D$ of $K_{5,n}$ is a {\em
  Zarankiewicz drawing} if the white vertices can be partitioned into
two sets, of sizes $\floor{n/2}$ and $\ceil{n/2}$, so that vertices in different
sets are antipodal in $D$, and vertices $a_i, a_j$ in the same set
satisfy $\ucr{D}{a_i,a_j}=4$ (see Figure~\ref{fig:kmnzar} for a
Zarankiewicz drawing of $K_{5,6}$).   A quick calculation shows
that every Zarankiewicz drawing of $K_{5,n}$ is an optimal drawing.

\begin{theorem}[Decomposition of optimal drawings of $K_{5,n}$, for
    $n$ even]
\label{thm:main2}
Let $D$ be an optimal drawing of $K_{5,n}$, with $n$ even. Then
the set of $n$ white vertices can be
  partitioned into two sets $A,B$ (one of which may be empty), with
  $|A|=4t$ for some nonnegative integer $t$, such that:
 (i) the vertices in $B$ can be decomposed into $|B|/2$
  antipodal pairs;  and
(ii) the drawing of $K_{5,4t}$ induced by $A$ is antipodal-free, and
  it is isomorphic to the drawing $D_{r,s}$ described in
  Section~\ref{sec:spdr}, for some integers  $r,s$ such that $r+s=t$.
Equivalently, either $D$ is the superimposition of Zarankiewicz
drawings, or it can be obtained by superimposing Zarankiewicz
drawings to the drawing $D_{r,s}$ described in Section~\ref{sec:spdr}, for
some integers  $r,s$ (see Figure~\ref{fig:anti}). 
\end{theorem}

\begin{proof}
We proceed by induction on $n$. 
It is trivial to check that the two
white vertices of every
optimal drawing of $K_{5,2}$ are an antipodal pair, and so the
statement holds in the base case $n=2$. For the inductive step, we consider an even
integer $n$, and assume that the
statement is true for all $k < n$. 

Let $D$ be an optimal drawing of $K_{5,n}$.
If $D$ has no antipodal pairs, then the statement follows
immediately from Theorem~\ref{thm:main1} (without even using the
induction hypothesis). Thus we may assume that $D$
has at least one antipodal pair $a_i, a_j$.
It suffices to show that the drawing $D'$ that results by removing
$a_i$ and $a_j$ from $D$
is an optimal drawing of $K_{5,n-2}$, as then
the result follows by the induction hypothesis.  Clearly $\cru(D) =
\cru(D') + \sum_{k\in {\tred{[n]}}-\{i,j\}} (\ucr{D}{a_i,a_k}+\ucr{D}{a_j,a_k})
\ge \cru(D') + (n-2)Z(5,3) = \cru(D')+4n-8$.  Thus $\cru(D') \le
\cru(D) - 4n+8 = Z(5,n) - 4n+8$. An elementary calculation shows that
$Z(5,n) -4n+8=Z(5,n-2)$, so we obtain $\cru(D') \le Z(5,n-2)$. Since
$\cru(K_{5,n-2}) = Z(5,n-2)$, it follows that $\cru(D') =Z(5,n-2)$,
that is, $D'$ is an optimal drawing of $K_{5,n-2}$. 
\end{proof}


\section[Clean drawings]{Clean
  drawings.}\label{sec:clean}

A good drawing of $K_{5,n}$ is {\em clean} if:
\begin{enumerate} 
\item\label{pro:pfi} for all distinct white vertices $a_i,a_j$ such that
$\rot{D}{a_i} = \rot{D}{a_j}$, we have 
 $\ucr{D}{a_i,a_j} = 4$;
\item\label{pro:pfu} for all distinct white vertices 
$a_i, a_j, a_k, a_\ell$ such that
$\rot{D}{a_i} = \rot{D}{a_j}$ and $\rot{D}{a_k} =
  \rot{D}{a_\ell}$, we have $\ucr{D}{a_i,a_k} = \ucr{D}{a_j,a_\ell}$;
  and
\item\label{pro:pfo} for any distinct white
vertices $a_i, a_k$, 
$\ucr{D}{a_i,a_k} \le 4$.
\end{enumerate}

\begin{proposition}\label{pro:same}
Let $D$ be an optimal drawing of $K_{5,n}$. Then
there is an optimal drawing $D'$, isomorphic to $D$, that is clean.
\end{proposition}

\begin{proof}
For each white vertex $a_i$, define $d_i:=\sum_{\{a_\ell\,|\, \rot{D}{a_\ell}\neq
   \rot{D}{a_i}\}}
\ucr{D}{a_i,a_\ell}$. 
 Let $\pi\in\Rot{D}$. Take a white vertex 
$a_i$ with $\rot{D}{a_i}=\pi$,  such that for all
$j$ with \tred{$\rot{D}{a_j}= \pi$} we have $d_i \le d_j$. It is easy to see
that we can move 
every vertex $a_j$ with \tred{$\rot{D}{a_j}= \pi$} very close to $a_i$, so that
$\ucr{D}{a_i,a_k} = \ucr{D}{a_j,a_k}$ for every white vertex
$a_k\notin\{a_i,a_j\}$, and so that $\ucr{D}{a_i,a_j}=4$. If we
perform this
procedure for every rotation in $\Rot{D}$, the result is an optimal drawing
$D'$, isomorphic to $D$, that satisfies (\ref{pro:pfi}) and \eqref{pro:pfu}.

Now to prove that $D'$ also satisfies (\ref{pro:pfo}) we suppose, by way of contradiction, that 
there exist $a_i, a_k$ such that
$\ucr{D}{a_i,a_k} > 4$. Define $d_i, d_k$ as in the previous
paragraph. We may assume without loss of generality that 
$d_i \le d_k$. Now let $D''$ be the
drawing that results from moving $a_k$ very close to $a_i$, making it
have the same rotation as $a_i$, and so that $\ucr{D''}{a_i,a_\ell} =
\ucr{D''}{a_k,a_\ell}$ for every $\ell\not\in\{i,k\}$, and $\ucr{D''}{a_i,a_k}
  = 4$. It is readily checked that $D''$ has fewer crossings than $D'$,
  contradicting the optimality of $D'$.
\end{proof}

\begin{remark}\label{rem:clean}
We are interested in classifying optimal drawings up to
isomorphism (Theorem~\ref{thm:main1}). In view of Proposition~\ref{pro:same}, we may assume 
that all drawings of $K_{5,n}$ under
consideration are clean.
We will work under this assumption for the rest of the paper.
\end{remark}

\section[The key of a clean drawing]{The key  of a clean drawing.}\label{sec:keyscores}



We now associate to every clean drawing of $K_{5,n}$ an edge-labeled
graph that (as we will see) captures all its relevant crossing number
information.

Let $D$ be a clean drawing of $K_{5,n}$.  The {\em key} $\Key{D}$
of $D$ is the (edge-labeled) complete graph whose vertices are the
elements of $\Rot{D}$, and where each edge is labeled according to the
following rule: if $\pi,\pi'\in \Rot{D}$, with $\rot{D}{a_i}=\pi$ and
$\rot{D}{a_j}=\pi'$, then the label of the edge joining $\pi$ and
$\pi'$ is $\ucr{D}{a_i,a_j}$.  It follows from the cleanness of $D$ 
that $\ucr{D}{a_i,a_j}$ does not depend on the
choice of $a_i$ and $a_j$, and so $\Key{D}$ is well-defined for every
clean drawing $D$.  Moreover, it also follows that every edge label
in $\Key{D}$ is in $\{0,1,2,3,4\}$. 
The {\em core} of $D$ is the subgraph
$\PKey{D}$ of $\Key{D}$ that consists of all the vertices of $\Key{D}$
and the edges of $\Key{D}$ with label $1$. 
In Figure~\ref{fig:key} we give a (clean and optimal) drawing $D$ of $K_{5,3}$, and
illustrate its key and its core.

Our main interest is in
antipodal-free drawings, that is, those drawings in which every edge
label in $\Key{D}$ is in $\{1,2,3,4\}$.
A key is {\em $0$--free} (respectively, $4$-{\em free}) if none of
its edges has $0$ (respectively, $4$) as
a label. A key is {\em $\{0,4\}$-free} if it is both $0$- and
$4$-free.

	\begin{figure}[h!]
	\begin{center}
		\scalebox{0.7}{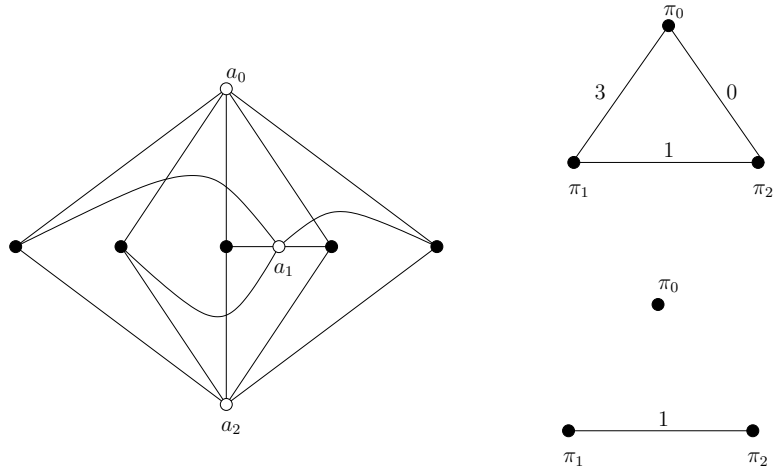}
		 \caption{\small{A drawing $D$ of
                     $K_{5,3}$. By letting $\rot{D}{a_0}=\pi_0, \rot{D}{a_1}=\pi_1$, and $\rot{D}{a_2}=\pi_2$,
                     we obtain the key $\Key{D}$ (right, above) and
                     the core $\PKey{D}$ (right, below) of $D$.}}
		 \label{fig:key}
	\end{center}
	\end{figure}




The main step in our strategy to understand optimal drawings is to
characterize which labelled graphs are the key of some optimal
drawing. To this end, we introduce a system of linear equations
associated to each key, as follows.

\begin{definition}[The system of linear equations of a key]\label{def:linsys}
Let $D$ be an optimal drawing of $K_{5,n}$, with $n$ even.  Let the
vertices of $\Key{D}$ (that is, the elements of $\Rot{D}$) be labelled $\pi_0, \pi_1, \ldots, \pi_{m-1}$, and
let $\lambda_{ij}$ denote the label of the edge $\pi_i \pi_j$,
for all $i \neq j$. For each $i\in {\tred{[m]}}$, the {\em linear equation
  $E(\pi_i,\Key{D})$ for
  $\pi_i$ in $\Key{D}$} is the linear equation on the variables $t_0,
t_1, \ldots, t_{m-1}$ given by
\[
E(\pi_i,\Key{D})\hbox{\hglue 0.3 cm} : \hbox{\hglue 0.3 cm}
2t_i + \sum_{j\in {\tred{[m]}},\,j\neq i} (\lambda_{ij} -2 ) t_j = 0.
\]

The set $\{E(\pi_i,\Key{D})\}_{i\in {\tred{[m]}}}$ is the {\em system of linear
  equations associated to} $\Key{D}$, and is denoted $\LL(\Key{D})$.
\end{definition}

The characterization of when a labelled graph is the key of an
optimal drawing is mainly based on the following crucial fact.

\begin{proposition}\label{pro:tempev}
Let $D$ be an optimal drawing of $K_{5,n}$, with $n$ even. Then
the system of linear equations $\LL(\Key{D})$ associated to $\Key{D}$
has a positive integral solution $(t_0,t_1,\ldots,t_{m-1})$ such that
$t_0 + t_1 + \cdots + t_{m-1} = n$. 
\end{proposition}

\begin{proof}
First we show that if $D$ is an optimal drawing of $K_{5,n}$ with $n$
even, then for every $i=0,1,\ldots,n-1$, we have $\ucr{D}{a_i} =
2n-4$. To this end, suppose that $\ucr{D}{a_i} > 2n-4$ for some
$i$. Since $D$ is optimal, $\cru(D) = Z(5,n) = n(n-2)$, and so the
drawing $D'$ of $K_{5,n-1}$ that results by removing $a_i$ from $D$
has fewer than $n(n-2) - (2n-4) = n^2 - 4n +4 = (n-2)^2 = Z(5,n-1)$
crossings, contradicting that $\cru(K_{5,n-1}) = Z(5,n-1)$. Thus
  $\ucr{D}{a_i} \le 2n-4$ for every $i$. Now suppose that
  $\ucr{D}{a_i}<2n-4$ for some $i$. Then $\cru(D) =
  (1/2)\sum_{j\in {\tred{[n]}}} \ucr{D}{a_j} < (1/2)(2n-4)n =
  n(n-2)$, contradicting that $\cru(K_{5,n}) = Z(5,n) = n(n-2)$. Thus
    for every $i\in {\tred{[n]}}$ we have $\ucr{D}{a_i} = 2n-4$, as
    claimed. 

Now let $\pi_0,\pi_1,\ldots,\pi_{m-1}$ be the elements of $\Rot{D}$
(that is, the vertices of $\Key{D}$), and for each $i,j\in {\tred{[m]}}, i\neq
j$, let $\lambda_{ij}$ denote the label of the edge $\pi_i\pi_j$ in $\Key{D}$. 
For each $i\in{\tred{[m]}}$, let $t_i$ be
the number of vertices with rotation $\pi_i$ in $D$. Then (using that
$D$ is clean) for
every $i\in {\tred{[m]}}$ and 
every white vertex $a_k$ with $\rot{D}{a_k}=\pi_i$  we have $\ucr{D}{a_k}
= 4(t_i-1) + \sum_{j\in {\tred{[m]}},j\neq i}
\lambda_{ij} t_j$. Now from the previous paragraph for each
$a_k$ we have $\ucr{D}{a_k} = 2n-4$. Using that
$n=\sum_{j\in {\tred{[m]}}} t_j$,  we obtain
$4(t_i-1) + \sum_{j\in {\tred{[m]}},j\neq i}
\lambda_{ij}t_j  = 2\sum_{j\in {\tred{[m]}}} t_j -
4$. Equivalently, 
$2t_i + \sum_{j\in {\tred{[m]}},j\neq i} (\lambda_{ij} -2 )
t_j = 0$, for every $i\in {\tred{[m]}}$.  Thus $(t_0,t_1,\ldots,t_{m-1})$ is a positive integral
solution of $\LL(\Key{D})$.
\end{proof}

\section{Properties of the key of a clean drawing.}\label{sec:someprco}


We start with an easy, yet crucial, observation.

\begin{proposition}\label{pro:at4}
Let $D$ be an optimal drawing of $K_{5,n}$. Then, for any three
distinct white vertices $a_i,a_j,a_k$, $\ucr{D}{a_i,a_j} +\ucr{D}{a_j,a_k}
+\ucr{D}{a_i,a_k}$ is an even number greater than or equal to $4$.
\end{proposition}

\begin{proof}
This follows since $\cru(K_{5,3}) = Z(5,3) = 4$ and (see for
instance~\cite{kleitman}) every good drawing of $K_{5,3}$ has an even
number of crossings.
\end{proof}

The following is an equivalent form of this statement, in the setting
of keys.

\begin{proposition}\label{pro:forkeys}
Let $D$ be a clean drawing of $K_{5,n}$, and let $\pi_0, \pi_1, \pi_2$
be vertices of $\Key{D}$. Let $\lambda_{ij}$ be the label of the edge
$\pi_i \pi_j$, for $i,j\in\{0,1,2\}, i\neq j$. Then $\lambda_{01} +
\lambda_{12} + \lambda_{02}$ is an even number greater than or equal to
$4$.\hfill$\Box$ 
\end{proposition}

Let $\gamma, \kappa$ be cyclic permutations on the same set of symbols. A
{\em route} from $\gamma$ to $\kappa$ is a set of distinct
transpositions, which may be ordered into some sequence such that the
successive application of (all) the transpositions in this sequence takes
$\gamma$ to $\kappa$. For instance, if $\gamma=(abcd)$ and
$\kappa=(acdb)$, then $\{(bd),(bc)\}$ is a route from $\gamma$ to
$\kappa$: if we apply first $(bc)$ to $\gamma$, and then $(bd)$ to the
resulting cyclic permutation, we obtain $\kappa$. 

The {\em size} $|P|$ of a route $P$ is its number of
transpositions.   An
{\em antiroute} from $\gamma$ to $\kappa$ is a route from $\gamma$ to
the reverse cyclic permutation $\overline{\kappa}$ of $\kappa$. 
Note that if $P$ is a route (respectively, antiroute) from $\gamma$ to $\kappa$, then
$P$ is also a route (respectively, antiroute) from $\kappa$ to $\gamma$. The {\em
  antidistance} between two cyclic permutations is the smallest size of an
antiroute between them.

The following is an easy consequence of (the proof of) Theorem 5
in~\cite{woodall}.

\begin{lemma}\label{lem:wo}
Let $D$ be a good drawing of $K_{5,2}$, with white vertices $a_0,
a_1$. Then there is an antiroute
from $\rot{D}{a_0}$ to $\rot{D}{a_1}$ of size 
$\ucr{D}{a_0,a_1}$.\hfill$\Box$
\end{lemma}


The following statement is implicitly proved in
the discussion after the proof of~\cite[Theorem 5]{woodall}.

\begin{lemma}\label{lem:woo}
Let $D$ be a clean drawing of $K_{5,r}$ with white vertices $a_0, a_1,
\ldots,$ $ a_{r-1}$, and let $\pi_i:=\rot{D}{a_i}$. Suppose that
$\pi_i\neq \pi_j$ whenever $i\neq j$, and for all $i\neq j$ let
$\lambda_{ij}:=\ucr{D}{a_i,a_j}$. For $k=0,1,2,3,4$, let
$\gamma_k:=\rot{D}{k}$. Then
there exist:
\begin{enumerate}

\item\label{it:w1} for all $i,j\in [r]$ with $i\neq j$, an antiroute
  $P_{ij}$ from $\pi_i$ to $\pi_j$ of size $\lambda_{ij}$; 
\item\label{it:w2} for all $k,\ell \in [5]$ with $k\neq
  \ell$, an antiroute $Q_{k\ell}$ from $\gamma_k$ to $\gamma_\ell$;
\end{enumerate}
such that the transposition $(a_i\,a_j)$ is in $Q_{k\ell}$ if and only
if the transposition $(k\,\ell)$ is in $P_{ij}$.\hfill$\Box$
\end{lemma}

We now use these powerful statements to prove that certain graphs
cannot be the subgraphs of the key of a clean drawing.

\begin{proposition}\label{pro:4claw}
The graph in Figure~\ref{fig:sk13} is not the key of any clean drawing
of $K_{5,n}$.
\end{proposition}

\begin{figure}
\begin{center}
\scalebox{0.7}{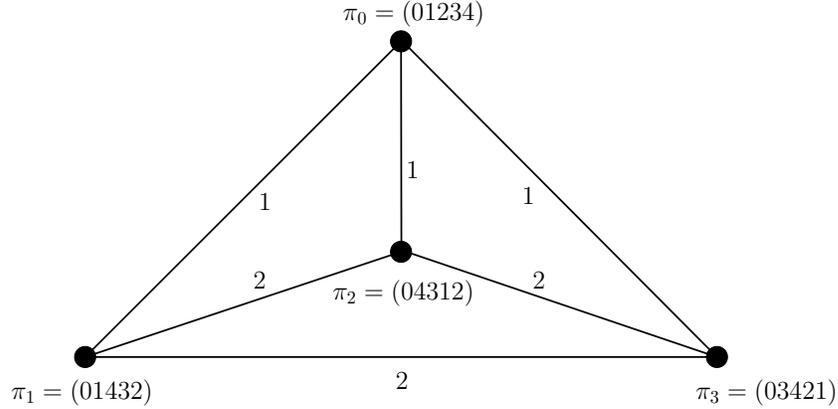}
\end{center}
\caption{This cannot be the key of a clean drawing of $K_{5,n}$.}
\label{fig:sk13}
\end{figure}

\begin{proof}
Suppose by way of contradiction that the graph in
Figure~\ref{fig:sk13} is the key of some clean drawing of $K_{5,n}$.  
This implies in particular that there exists a
drawing $D$ of $K_{5,4}$ with white vertices $a_0, a_1, a_2, a_3$ such
that $\rot{D}{a_i}=\pi_i$ for $i=0,1,2,3$, with $\pi_0=(01234),
\pi_1=(01432), \pi_2=(04312)$, and $\pi_3=(03421)$,
and $
\ucr{D}{a_0,a_1} =
\ucr{D}{a_0,a_2} = 
\ucr{D}{a_0,a_3} = 1$, and
$\ucr{D}{a_1,a_2} = 
\ucr{D}{a_1,a_3} = 
\ucr{D}{a_2,a_3} = 2$.

The required contradiction is obtained by showing that there do not exist rotations 
$\rot{D}{0},
\rot{D}{1},
\rot{D}{2},
\rot{D}{3},
\rot{D}{4}$, and antiroutes $P_{ij}, Q_{k\ell}$ that satisfy
Lemma~\ref{lem:woo} (with the given values of $\ucr{D}{a_i,a_j}$ for
$i,j\in\{0,1,2,3\},$ $ i\neq j$). We start by determining the possible antiroutes
$P_{ij}$ (these depend only on the information we already have). Then
we investigate the possible antiroutes $Q_{k\ell}$ consistent with
each choice of the antiroutes $P_{ij}$, and prove that, in all cases,
every possible choice of
$\rot{D}{0},
\rot{D}{1},
\rot{D}{2},$ $
\rot{D}{3}$ and 
$\rot{D}{4}$ leads to an inconsistency.

The following facts are easily verified:
(i) the only antiroute from $\pi_0$ to $\pi_1$ of
size $1$ is
$\{(01)\}$; (ii) the only antiroute from $\pi_0$ to $\pi_2$ of size $1$ is
$\{(12)\}$; (iii)  the only antiroute from $\pi_0$ to $\pi_3$ of size
$1$ is
$\{(34)\}$; (iv)
the only antiroute of size $2$ from $\pi_1$ to $\pi_2$ is
$\{(02),(34)\}$; (v)
there are two distinct antiroutes of size $2$ from $\pi_2$ to
$\pi_3$, namely $\{(01),(02)\}$ and $\{(03),(04)\}$; and (vi) there
are two
distinct antiroutes of size $2$ from $\pi_1$ to $\pi_3$, namely
$\{(02),(12)\}$ and $\{(23),(24)\}$.

Now for $i,j\in\{0,1,2,3\}, i\neq j$, let $P_{ij}$ be the antiroute
guaranteed by Lemma~\ref{lem:woo}. By the previous observations it
follows that necessarily 
$P_{01}=\{(01)\}$, $P_{02}=\{(12)\}$,
$P_{03}=\{(34)\}$, and $P_{12} = \{(02),(34)\}$.  Also by the previous
observations there are two choices for $P_{23}$, namely
$\{(01),(02)\}$ and $\{(03),(04)\}$; and there are two choices for
$P_{13}$, namely $\{(02),(12)\}$ and $\{(23),(24)\}$.

Thus $P_{01}, P_{02}, P_{03}, P_{12}$ are all determined:

\[P_{01}=\{(01)\}, P_{02}=\{(12)\},
P_{03}=\{(34)\}, P_{12} = \{(02),(34)\},
\]

\medskip

and there
are four possible combinations of $P_{13}$ and $P_{23}$:
\begin{itemize}
\item[(a)] $P_{23} = \{(01),(02)\}$ and $P_{13} = \{(02),(12)\}$. 
\medskip

\noindent In this case, by Lemma~\ref{lem:woo}, we have 
$
Q_{01}= \{(a_0a_1),(a_2a_3)\},
Q_{02}= \{(a_1a_2),(a_2a_3),(a_1a_3)\},
Q_{03}= \emptyset,
Q_{04}= \emptyset,
Q_{12}= \{(a_0a_2),(a_1a_3)\},$ $
Q_{13}= \emptyset,
Q_{14}= \emptyset,
Q_{23}= \emptyset,
Q_{24}= \emptyset$, and
$Q_{34} = \{(a_0a_3),(a_1a_2)\}$.
\medskip

\item[(b)] $P_{23} =\{(01),(02)\}$ and $P_{13} = \{(23),(24)\}$.
\medskip

\noindent In this case, by Lemma~\ref{lem:woo}, we have 
$
Q_{01}= \{(a_0a_1),(a_2a_3)\},
Q_{02}= \{(a_1a_2),(a_2a_3)\},
Q_{03}= \emptyset,
Q_{04}= \emptyset,
Q_{12}= \{(a_0a_2)\},
Q_{13}= \emptyset,
Q_{14}= \emptyset,
Q_{23}= \{(a_1a_3)\},
Q_{24}= \{(a_1a_3)\}$, and
$Q_{34} = \{(a_0a_3),(a_1a_2)\}$.
\medskip

\item[(c)] $P_{23} = \{(03),(04)\}$ and $P_{13} = \{(02),(12)\}$.
\medskip

\noindent In this case, by Lemma~\ref{lem:woo}, we have 
$
Q_{01}= \{(a_0a_1)\},
Q_{02}= \{(a_1a_2),$ $(a_1a_3)\},
Q_{03}= \{(a_2a_3)\},
Q_{04}= \{(a_2a_3)\},
Q_{12}= \{(a_0a_2),(a_1a_3)\},$ $
Q_{13}= \emptyset,
Q_{14}= \emptyset,
Q_{23}= \emptyset,
Q_{24}= \emptyset$, and
{$Q_{34} = \{(a_0a_3),(a_1a_2)\}$.}
\medskip

\item[(d)] $P_{23} = \{(03),(04)\}$ and $P_{13} = \{(23),(24)\}$.
\medskip

\noindent In this case, by Lemma~\ref{lem:woo}, we have 
$
Q_{01}= \{(a_0a_1)\},
Q_{02}= \{(a_1a_2)\},$ $
Q_{03}= \{(a_2a_3)\},
Q_{04}= \{(a_2a_3)\},
Q_{12}= \{(a_0a_2)\},
Q_{13}= \emptyset,
Q_{14}= \emptyset,
Q_{23}= \{(a_1a_3)\},
Q_{24}= \{(a_1a_3)\}$, and
{$Q_{34} = \{(a_0a_3),(a_1a_2)\}$.}
\end{itemize}
\bigskip

We only analyze (that is, derive a contradiction from) (a). The cases
(b), (c), and (d) are handled in a totally analogous manner.

Since $Q_{13}=Q_{14}=\emptyset$,  it follows
that $\rot{D}{3}$ and $\rot{D}{4}$ are both equal to the reverse of
$\rot{D}{1}$; in particular,
$\rot{D}{3} = \rot{D}{4}$.
Since $Q_{01}=\{(a_0a_1),(a_2a_3)\}$ and $Q_{12}=\{(a_0a_2),(a_1a_3)\}$, it follows
that in $\rot{D}{1}$: (i) $a_0$ and $a_1$ must be adjacent; (ii) $a_2$ and $a_3$
must be adjacent; (iii) $a_0$ and $a_2$ must be adjacent; and (iv) $a_1$ and $a_3$ must
be adjacent. It follows immediately that $\rot{D}{1}$ is either
$(a_0a_2a_3a_1)$ or $(a_0a_1a_3a_2)$. Since $\rot{D}{3}$ and $\rot{D}{4}$ are both the 
reverse of $\rot{D}{1}$, then each of $\rot{D}{3}$ and
$\rot{D}{4}$ is either $(a_0a_1a_3a_2)$ or $(a_0a_2a_3a_1)$. However, since
$Q_{34}=\{(a_0a_3),(a_1a_2)\}$, then one must reach the reverse of $\rot{D}{4}$ from
$\rot{D}{3}$ by applying the transpositions $(a_0a_3)$ and $(a_1a_2)$ (in some
order). Since neither of these transpositions may be applied to $(a_0a_1a_3a_2)$
or $(a_0a_2a_3a_1)$, we obtain the required contradiction.
\end{proof}

\begin{proposition}\label{pro:chimuelo}
The graph in Figure~\ref{fig:chim} is not the key of any clean drawing
of $K_{5,n}$.
\end{proposition}

\begin{figure}
\begin{center}
\scalebox{0.7}{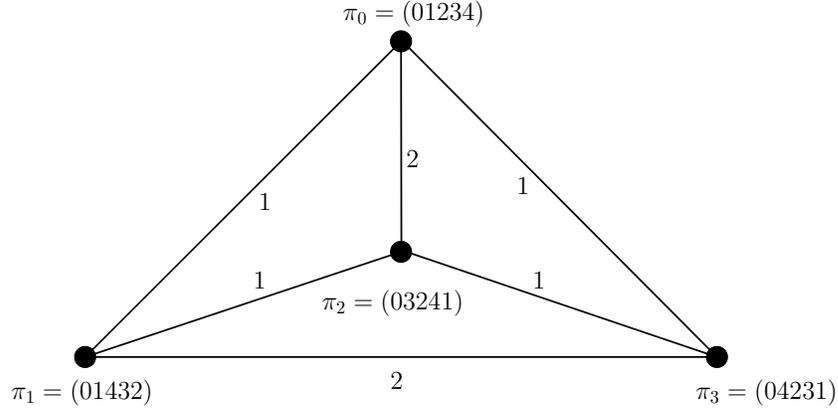}
\end{center}
\caption{This cannot be the key of a clean drawing of
  $K_{5,n}$.}
\label{fig:chim}
\end{figure}

\begin{proof}
Suppose by way of contradiction that the graph in
Figure~\ref{fig:chim} is the key of some clean drawing of $K_{5,n}$.  
Thus there exists a drawing $D$ of $K_{5,4}$ with white vertices 
$a_0, a_1, a_2, a_3$ such
that $\rot{D}{a_i}=\pi_i$ for $i=0,1,2,3$, with $\pi_0=(01234),
\pi_1=(01432), \pi_2=(03241)$, and $\pi_3=(04231)$,
and $\ucr{D}{a_0,a_1} =
\ucr{D}{a_1,a_2} = 
\ucr{D}{a_2,a_3} = 
\ucr{D}{a_0,a_3}=1$, and
$\ucr{D}{a_0,a_2} = 
\ucr{D}{a_1a_3} = 2$. 
For $i,j\in \{0,1,2,3\},i\neq j$, let $P_{ij}$ be the antiroute
guaranteed by Lemma~\ref{lem:woo}.  It is easy to verify that the only
antiroute of size $1$ from $\pi_0$ to $\pi_1$ is $\{(01)\}$, and so
necessarily $P_{01}=\{(01)\}$. Analogous arguments show that necessarily
$P_{23}=\{(01)\}$ and that $P_{12}=P_{03}=\{(23)\}$. 
It is also readily checked that there are two 
antiroutes of size $2$ from $\pi_0$ to $\pi_2$, namely 
$\{(04),(14)\}$
and 
$\{(24),(34)\}$ (moreover, these are also the two antiroutes of size
$2$ from $\pi_1$ to $\pi_3$). Thus
each of $P_{02}$ and $P_{13}$ is either
$\{(04),(14)\}$
or 
$\{(24),(34)\}$. 

Thus $P_{01},P_{03},P_{12}$, and $P_{23}$ are all determined:
\[
P_{01}=P_{23}=\{(01)\}, P_{03}=P_{12}=\{(23)\},
\]
and there are four possible combinations of $P_{02}$ and $P_{13}$:
\begin{itemize}

\medskip
\item[(a)]
$P_{02}=P_{13}=\{(04),(14)\}$.
\\

\noindent
In this case, by Lemma \ref{lem:woo}, 
$Q_{01}=\{(a_0a_1),(a_2a_3)\},
Q_{04}=\{(a_0a_2),$ $(a_1a_3)\},$ $
Q_{14}=\{(a_0a_2),(a_1a_3)\},
Q_{23}=\{(a_0a_3),(a_1a_2)\},$ and
$Q_{02}=Q_{03}=Q_{12}=Q_{13}=Q_{24}=Q_{34}=\emptyset$.

\medskip
\item[(b)] 
$P_{02}=\{(04),(14)\}$ and 
$P_{13}=\{(24),(34)\}$.\\

\noindent
In this case, by Lemma \ref{lem:woo}, 
$Q_{01}=\{(a_0a_1),(a_2a_3)\},
Q_{04}=Q_{14}=\{(a_0a_2)\},$ $
Q_{23}=\{(a_0a_3),(a_1a_2)\},
Q_{24}=Q_{34}=\{(a_1a_3)\}$, and
$Q_{02} = Q_{03} = Q_{12} = Q_{13} = \emptyset$.

\medskip
\item[(c)]  
$P_{02}=\{(24),(34)\}$ and 
$P_{13}=\{(04),(14)\}$.\\

\noindent
In this case, by Lemma~\ref{lem:woo},
$Q_{01}=\{(a_0a_1),(a_2a_3)\}$,
$Q_{04}=Q_{14}=\{(a_1a_3)\}$,
$Q_{23}=\{(a_0a_3),(a_1a_2)\}$,
$Q_{24}=Q_{34}=\{(a_0a_2)\}$,
and
$Q_{02}=Q_{03}=Q_{12}=Q_{13}=\emptyset$.

\medskip
\item[(d)]  
$P_{02}=P_{13}=\{(24),(34)\}$.\\

\noindent
In this case, by Lemma \ref{lem:woo}, 
$Q_{01}=\{(a_0a_1),(a_2a_3)\},
Q_{23}=\{(a_0a_3),$ $(a_1a_2)\},
Q_{24}=Q_{34}=\{(a_0a_2), (a_1a_3)\}$, and
$Q_{02}=Q_{03}=Q_{04}=Q_{12}=Q_{13}=Q_{14}=\emptyset$.
\end{itemize}

\medskip

We only analyze (that is, derive a contradiction from) (a). The cases
(b), (c), and (d) are handled analogously.

Since $Q_{02}=Q_{03}=Q_{12}=Q_{13}=Q_{24}=Q_{34}=\emptyset$, it follows that $\rot{D}{2}$ and
$\rot{D}{3}$ are equal to each other, and equal to the reverse of
each of $\rot{D}{0}$, $\rot{D}{1}$, and 
 $\rot{D}{4}$. Thus 
$\rot{D}{0}=\rot{D}{1}=\rot{D}{4}$. Since 
$Q_{01}=\{(a_0a_1),(a_2a_3)\}$ and $Q_{04}=\{(a_0a_2),(a_1a_3)\}$, 
it follows that in $\rot{D}{0}$: (i) $a_0$ and $a_1$ must be adjacent;
(ii) $a_2$ and $a_3$ must be adjacent; (iii) $a_0$ and $a_2$ must be
adjacent; and (iv) 
$a_1$ and $a_3$ must be adjacent. Thus
$\rot{D}{0}$ is either $(a_0 a_2 a_3 a_1)$ or 
$(a_0 a_1 a_3 a_2)$. Now since
$Q_{23}=\{(a_0a_3),(a_1a_2)\}$, it follows that 
in $\rot{D}{2}$ (and hence in
its reverse $\rot{D}{0}$) we have that
$a_0$ is adjacent to $a_3$, and that  $a_1$
is adjacent to $a_2$.  But this is impossible, since in neither $(a_0 a_2 a_3
a_1)$ nor $(a_0 a_1 a_3 a_2)$ any of these
adjacencies occurs.
\end{proof}

\section{Properties of cores. I. Forbidden subgraphs.}\label{sec:someprrk1}

We recall that the {\em core} of a clean drawing $D$ of $K_{5,n}$ is the
subgraph 
$\PKey{D}$ of $\Key{D}$ that consists of all the vertices of $\Key{D}$
and the edges of $\Key{D}$ with label $1$. Note that while $\Key{D}$
is obviously connected, $\PKey{D}$ may be disconnected. 
As all edges of a
core are labelled $1$, we sometimes omit the reference to the edge
labels altogether when working with $\PKey{D}$.


Our first result on the structure of cores is a workhorse for the next
few sections.

\begin{claim}\label{cla:ch}
If $\pi_1, \pi_2$ and $\pi_3$ are distinct rotations for white vertices in a drawing of $K_{5,n}$, then there exists at
most one rotation $\pi_0$ such that there is an antiroute of size $1$
from $\pi_0$ to each of $\pi_1, \pi_2$, and $\pi_3$.
\end{claim}
\begin{proof}
\tred{
By way of contradiction, suppose that there exist distinct vertices $\pi_0,
\pi_1,\pi_2,\pi_3,\pi_4$ and antiroutes of size $1$
from $\pi_i$ to $\pi_1,\pi_2$, and $\pi_3$, for $i=0$ and
$4$. 
For $j=1,2,3$ the antiroutes from $\pi_0$ and $\pi_4$ to $\pi_j$ induce a
route $P_{04}(j)$ of size two from $\pi_0$ to $\pi_4$. Assume without loss of generality that $\pi_0=(01234)$.}
\tred{
Suppose that for some $j$, the transpositions in $P_{04}(j)$ involve (in total) four distinct elements in
$\{0,1,2,3,4\}$. It is immediately checked that this implies that
$P_{04}(j)$ is the only route of size $2$ from $\pi_0$ to $\pi_4$, and
that this in turn implies that at least two of $\pi_1, \pi_2$, and
$\pi_3$ are equal to each other, a contradiction. Thus each of $P_{04}(1),
P_{04}(2)$, and $P_{04}(3)$ involve fewer than four elements in
$\{0,1,2,3,4\}$. None of these routes can involve only two
elements (since they have size $2$, and $\pi_0\neq \pi_4$), and so we
conclude that each of $P_{04}(1), P_{04}(2)$, and $P_{04}(3)$ involve exactly three
elements in $\{0,1,2,3,4\}$. In particular, $P_{04}(1)$ must equal either $\{(k,k+1),(k,k+2)\}$
or $\{(k+1,k+2),$ $(k,k+2)\}$, for some
$j\in\{0,1,2,3,4\}$ (operations are modulo $5$; we note that we
deviate from the usual notation and separate the elements of a
transposition with a comma, for readability purposes). We derive a
  contradiction assuming that the first possibility holds; the other
  possibility is handled analogously.
Relabelling
$0,1,2,3$, and $4$, if needed, we may assume that
$P_{04}(1)=\{(01),(02)\}$. Thus $\pi_4$ is  $(03412)$. It is readily verified that the only routes of size
$2$ from $\pi_0=(01234)$ to $\pi_4=(03412)$ are $P_{04}(1)=\{(01),(02)\}$ and $\{(03),(04)\}$. This in turn
immediately implies that the antiroutes of size $1$ from $\pi_0$ to $\pi_1$,
$\pi_2$, and $\pi_3$ are either $\{(01)\}$ or $\{(04)\}$, since the transpositions
$(02)$ and $(03)$ cannot be applied to $\pi_0$. But then we arrive from $\pi_0$
to two elements in  $\{\overline{\pi_1},\overline{\pi_2},\overline{\pi_3}\}$ by applying the same transposition; that is, 
$\pi_i=\pi_j$ for some $i,j\in\{1,2,3\}$, $i\neq j$, a contradiction. }
\end{proof}

\begin{proposition}\label{pro:b}
Let $D$ be an optimal drawing of $K_{5,n}$.
Suppose that $\Key{D}$ is $\{0,4\}$-free. Then:
\begin{enumerate}
\item\label{it:c2}  $\PKey{D}$ does not contain $K_{2,3}$ as a subgraph.
\item\label{it:c1} $\PKey{D}$ has maximum degree at most $3$. 
\item\label{it:c3} $\PKey{D}$ does not contain as a subgraph the graph
obtained from $K_4$ by subdividing exactly once each of the edges in a
$3$-cycle (see Fig.~\ref{fig:k32}).
\end{enumerate}
\end{proposition}

\begin{figure}
\begin{center}
\includegraphics[width=3cm]{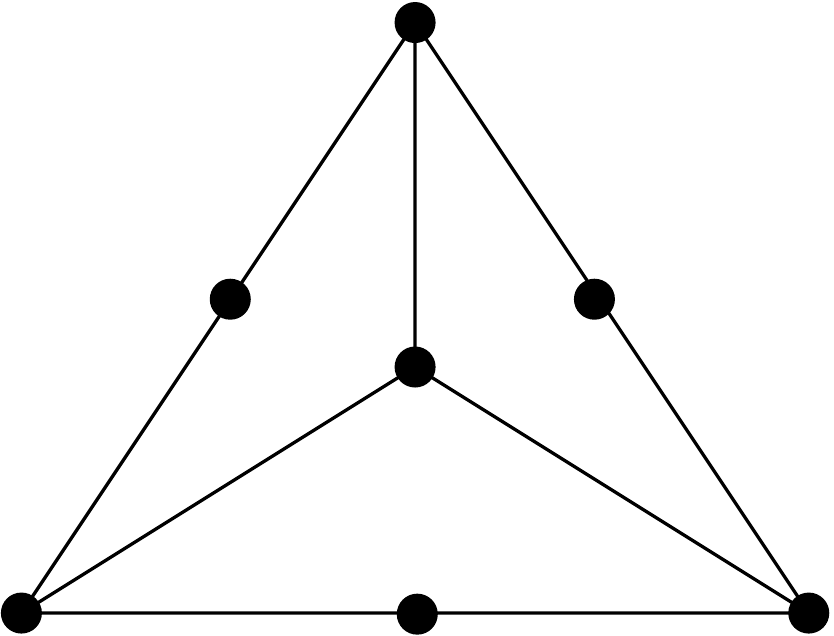}
\caption{The graph obtained by subdividing exactly once each of the
  edges in a $3$-cycle of $K_4$.}
\label{fig:k32}
\end{center}
\end{figure}

\begin{proof}
We start by noting that (\ref{it:c2}) follows immediately by 
Claim~\ref{cla:ch} and Lemma~\ref{lem:wo}.

Suppose now by way of contradiction that $\PKey{D}$ has a vertex $\pi_0$
of degree at least $4$. Thus $\PKey{D}$ has distinct vertices $\pi_1,
\pi_2,\pi_3,\pi_4$ such that the edge joining $\pi_0$ to $\pi_i$ has
label $1$, for $i=1,2,3,4$. Thus, for $i=1,2,3,4$, there exists an
antiroute from $\pi_0$ to $\pi_i$ of size $1$. \tred{Without loss of generality we may assume $\pi_0=(01234)$.}
The five cyclic rotations that have an antiroute of size
$1$ to $\pi_0$ are $(01432), (03214), (03421),(04312)$, and
$(04231)$. By performing a relabelling $j\to j+1$ on $\{0,1,2,3,4\}$ for some
$j\in\{0,1,2,3,4\}$ (with operations modulo $5$) if needed (note that 
the cyclic permutation $\pi_0=(01234)$ is left unchanged in such a relabelling),
we may assume without loss of generality that
\tred{$\{\pi_1,\pi_2,\pi_3,\pi_4\} = \{(01432), (03214),(03421),(04312)\}$.} By exchanging $\pi_1,\pi_2,\pi_3,\pi_4$ if needed, we may
assume that $\pi_1 =
(01432), \pi_2=(04312)$, and $\pi_3=(03421)$.

Since $\Key{D}$ is $\{0,4\}$-free, it
follows by Proposition~\ref{pro:forkeys} that the edge joining $\pi_i$ to $\pi_j$ has label $2$, for
$i,j\in\{1,2,3\}, i\neq j$.  Thus, for $i,j=1,2,3, i\neq j$, there exists an
antiroute from $\pi_i$ to $\pi_j$ of size $2$. Thus $\Key{D}$ contains as
 a subgraph  the graph in Figure~\ref{fig:sk13}, contradicting
Proposition~\ref{pro:4claw}. This proves (\ref{it:c1}).

We finally prove (\ref{it:c3}). 
Suppose by way of contradiction that $\PKey{D}$ contains as a subgraph
the graph obtained from $K_4$ by subdividing once each of the edges in
a $3$-cycle (Fig.~\ref{fig:k32}). Let $\rho_0$ be the ``central
vertex'' in Fig.~\ref{fig:k32}, that is, the only vertex in $\PKey{D}$
adjacent to three degree-$3$ vertices, and let $\rho_1, \rho_3, \rho_4$
denote these three vertices.
An argument similar to the one in the second paragraph of this proof
shows the following: if $\rho_0=(01234)$ is a vertex adjacent to
vertices $\rho_1,\rho_3,\rho_4$ in $\PKey{D}$, then 
we may assume (that is, perhaps after a relabelling of $0,1,2,3,4$),
that 
$\rho_1= (01432), \rho_3=(04231)$, and $\rho_4=(04312)$.
Now let $\rho_2$ be the vertex adjacent to
$\rho_1$ and $\rho_3$ in $\PKey{D}$. Thus it follows that in $\Key{D}$,
the edges joining $\rho_0$ and $\rho_1$, $\rho_0$
and $\rho_3$, $\rho_1$ and $\rho_2$, and $\rho_2$ and $\rho_3$ are labelled
$1$. 
By Proposition~\ref{pro:forkeys}, the edge joining $\rho_1$
  and $\rho_3$, as well as the edge joining $\rho_0$ and $\rho_2$  have  even labels, which must be $2$ since
$\Key{D}$ is $\{0,4\}$-free.
Now it is easy to verify that the only cyclic permutation other than
$\rho_0$ which has antiroutes of size $1$ to both $\rho_1$ and $\rho_3$
is
$(03241)$. Thus $\rho_2$ must be $(03241)$. 
But then the
subgraph of $\Key{D}$ induced by $\rho_0,\rho_1,\rho_2$, and $\rho_3$ 
is isomorphic to the graph in Figure~\ref{fig:chim}, contradicting
Proposition~\ref{pro:chimuelo}.
\end{proof}

\section{Properties of cores. II. Structural properties.}\label{sec:someprrk2}

\begin{proposition}\label{pro:c}
Let $D$ be an optimal drawing of $K_{5,n}$, with $n$
even. Suppose that $\Key{D}$ is \tred{$\{0,4\}$-free}. Then:
\begin{enumerate}
\item\label{it:c5} $\PKey{D}$ is bipartite.
\item\label{it:c4} $\PKey{D}$ is connected. 
\end{enumerate}
\end{proposition}
\begin{proof}
Suppose that $C=(\pi_0,\pi_1,\pi_2,\ldots,\pi_{r-1},\pi_r,\pi_0)$ is an
odd cycle in
$\PKey{D}$. 
It
follows from Proposition~\ref{pro:forkeys} that 
\tred{$\pi_0\pi_{2}$} must have an even label in $\Key{D}$, since $\pi_0\pi_1$ and
$\pi_1\pi_2$ are both labelled $1$ in $\Key{D}$; now this even label
must be $2$, since $\Key{D}$ is $\{0,4\}$-free. Similarly,  since
$\pi_2\pi_3$ and $\pi_3\pi_4$ are also labelled $1$ in $\Key{D}$, then
 $\pi_2\pi_4$ must also be labelled $2$ in $\Key{D}$. Now since both
$\pi_0 \pi_2$ and $\pi_2\pi_4$ have label $2$ in $\Key{D}$, it follows
that $\pi_0 \pi_4$ also has label $2$ in $\Key{D}$. 
By repeating this
argument we find that $\pi_0\pi_j$ must have label $2$ in
$\Key{D}$ for every even $j$. In particular, $\pi_0 \pi_r$ must have
label $2$, contradicting that $\pi_0\pi_r$ is in $\PKey{D}$ (that is,
that the label of $\pi_0\pi_r$ in $\Key{D}$ is $1$). Thus $\PKey{D}$
cannot have an odd cycle. This proves (\ref{it:c5}).

To prove (\ref{it:c4}) we assume, by way of contradiction, 
that $\PKey{D}$ is not connected.

We start by observing that $\Key{D}$ must have at least one edge
labelled $1$. Indeed, otherwise every edge  $\Key{D}$ has label of at
least $2$, and so $\cru{(D)} \ge 2{n\choose 2}
= n(n-1) > Z(5,n)$, contradicting the optimality of $D$. 

Thus there exists a component $H$ of $\PKey{D}$ with at least $2$
vertices. Let $U$ be the set of white vertices whose rotation is a
vertex in
$H$, and let $V$ be all the other white vertices. 
Let $r:=|U|$ and $s:=|V|$. 
Note that 
\begin{align}
\nonumber
\cru{(D)} &=
\sum_{\stackrel{a_i,a_j\in U,}{a_i\neq a_j}} \ucr{D}{a_i,a_j}
+
\sum_{\stackrel{a_i,a_j\in V,}{a_i\neq a_j}} \ucr{D}{a_i,a_j}
+
\sum_{a_i\in U, a_j\in V} \ucr{D}{a_i,a_j}\\
\label{eq:some1}
&\ge Z(5,r) + Z(5,s) + 2rs,
\end{align}
since every vertex of $U$ is joined to every vertex of $V$ by an edge
with a label $2$ or greater.

We claim that, moreover, strict inequality must hold in
\eqref{eq:some1}. To see this, first we note that, since $H$ has at least $2$
vertices, it follows that there exist white vertices $a_k,a_\ell$
whose rotations are
in $H$ and such that $\ucr{D}{a_k,a_\ell} = 1$. Since by assumption
$\PKey{D}$ is not connected, there is a vertex $\pi$ in $\PKey{D}$ not in
$H$. Let $a_i$ be a white vertex such that $\rot{D}{a_i} = \pi$. 
Now $\ucr{D}{a_k,a_i}$ and  $\ucr{D}{a_\ell,a_i}$ are both at least 
$2$. However, 
we cannot have $\ucr{D}{a_k,a_i}$ and \tred{$\ucr{D}{a_\ell,a_i}$} both equal to
$2$, since then $\ucr{D}{a_k,a_\ell}=1$ would contradict
Proposition~\ref{pro:at4}.
Thus either $\ucr{D}{a_k,a_i}$ or $\ucr{D}{a_\ell,a_i}$ is at least
$3$. This proves that  \tred{ Inequality \eqref{eq:some1}} must
be strict, that
is,
\begin{equation}
\cru{(D)} > Z(5,r) + Z(5,s) + 2rs.
\label{eq:some2}
\end{equation}

Suppose that $r$ (and consequently, also $s$) is even. 
In this case, 
since $Z(5,m)=m(m-2)$ for even $m$,
using \eqref{eq:some2} we obtain
$\cru{(D)} > r(r-2) + s(s-2) + 2rs = (r+s)(r+s-2)=Z(5,r+s)=Z(5,n)$, 
contradicting the optimality of $D$. 

Suppose finally that $r$ is odd (and so $s$
is odd, since $|U|+|V|=n$ is even). 
Using that $r$ and $s$ are odd, and that
$Z(5,m)=(m-1)^2$ for odd $m$, with \eqref{eq:some2} we obtain $\cru{(D)} 
> (r-1)^2 + (s-1)^2 + 2rs = (r+s)(r+s-2) + 2 = Z(5,r+s)+2 = Z(5,n)+2$,
again contradicting the optimality of $D$. This finishes the proof of
(\ref{it:c4}).
\end{proof}

\section{Properties of cores. III. Minimum degree.}\label{sec:someprrk3}

\begin{proposition}\label{pro:horse}
Let $D$ be an optimal drawing of $K_{5,n}$, with $n$
even. Suppose that $\Key{D}$ is $\{0,4\}$-free. 
Let  $\pi_0, \pi_1, \pi_2, \pi_3$ be a path in
$\PKey{D}$. Suppose that in \tred{$\PKey{D}$}, $\pi_1$ is the only vertex adjacent to both $\pi_0$ and
$\pi_2$, and $\pi_2$ is the only vertex adjacent to both $\pi_1$ and $\pi_3$. 
Then:
\begin{enumerate}
\item\label{it:bigone} every vertex in $\PKey{D}$ is adjacent (in
  $\PKey{D}$) to a vertex in
  $\{\pi_0,\pi_1,$ $\pi_2,\pi_3\}$; and
\item\label{it:bigtwo} $\pi_0$ and $\pi_3$ are adjacent in \tred{$\PKey{D}$}. 
\end{enumerate}
\end{proposition}

\begin{proof}
Let $\pi_0, \pi_1, \ldots, \pi_{r-1}$ be the vertices of $\PKey{D}$
(and of $\Key{D}$ as well). For $i,j\in [r], i\neq j$, let
$\lambda_{ij}$ denote the label of the edge that joins $\pi_i$ to
$\pi_j$ in $\Key{D}$.  Recall that $\PKey{D}$ is bipartite
(Proposition~\ref{pro:c}\eqref{it:c5}). Since
$\pi_0,\pi_1,\pi_2,\pi_3$ is a path in $\Key{D}$, it follows that
$\pi_0$ and $\pi_2$ are in the same chromatic class $A$, and $\pi_1$
and $\pi_3$ are in the same chromatic class $B$. Moreover, since
$\Key{D}$ is $\{0,4\}$-free, it follows from Proposition~\ref{pro:forkeys}
that $\lambda_{ij}=2$ whenever
 $\pi_i$ and $\pi_j$ belong to the same chromatic class.
Thus we have $\lambda_{02}=\lambda_{13}=2$ and (since
$\pi_0,\pi_1,\pi_2,\pi_3$ is a path in $\PKey{D}$)
$\lambda_{01}=\lambda_{12}=\lambda_{23}=1$. It follows that the
equations of $\LL(\Key{D})$ corresponding to $\pi_0,\pi_1,\pi_2$, and
$\pi_3$ are:
\begin{align*}
\begin{matrix}
{\small
\begin{tabular}{c r c r c r c r c r c r}
  $E_0$  : &  $2t_0 $ & $-$ & $ t_1$ &      &       & $+$ &
  $(\lambda_{03}-2)t_3$ & $+$ & $ \sum\limits_{j\in [r], j>3} (\lambda_{0j}-2) t_j$ & $=$ & $0$, \\
  $E_1$  : &  $-t_0 $ & $+$ & $2t_1$ & $-$  & $t_2$ &     &       &
  $+$ & $ \sum\limits_{j\in[r], j>3} (\lambda_{1j}-2) t_j$ & $=$ & $0$, \\
  $E_2$  : &          & $-$ & $ t_1$ & $+$  &$2t_2$ & $-$ & $t_3$ & $+$ & $ \sum\limits_{j\in[r],j>3} (\lambda_{2j}-2) t_j$ & $=$ & $0$, \\
  $E_3$  : &  $ (\lambda_{03}-2)t_0 $ &     &        & $-$  & $t_2$ & $+$ &$2t_3$ & $+$ & $ \sum\limits_{j\in[r],j>3} (\lambda_{3j}-2) t_j$ & $=$ & $0$, \\
\end{tabular}
}
\end{matrix}
\end{align*}

\noindent where for simplicity we define $E_i:=E(\pi_i,\Key{D})$ for
$i\in \{0,1,2,3\}$. 
Summing up these
four linear equations we obtain
\begin{equation}\label{eq:l16}
(\lambda_{03}-1)t_0+ (\lambda_{03}-1)t_3+\sum_{j\in[r],j>3} (\lambda_{0j}+\lambda_{1j}+\lambda_{2j}+\lambda_{3j}-8) t_j=0
\end{equation}
We claim all the coefficients in \eqref{eq:l16} are nonnegative. First
we note that since $\lambda_{03} \ge 1$, then the coefficients of
$t_0$ and $t_3$ are indeed nonnegative. For the remaining coeficients,
consider any vertex $\pi_j$ in $\Key{D}$, with $j>3$. Since $\Key{D}$
is $\{0,4\}$-free, it follows that $\lambda_{ij}\ge 1$ for every $i\in
\{0,1,2,3\}$.

Since $\PKey{D}$ is bipartite, it follows that $\pi_j$ cannot be
adjacent (in $\PKey{D}$) to two elements in
$\{\pi_0,\pi_1,\pi_2,\pi_3\}$ whose indices have distinct parity. Now
it follows by hypothesis that $\pi_j$ cannot be adjacent to both
$\pi_0$ and $\pi_2$, or to $\pi_1$ and $\pi_3$. Thus $\pi_j$ is
adjacent to at most  one of $\pi_0,\pi_1,\pi_2$ and $\pi_3$ in
$\PKey{D}$. Using this, and the fact that $\pi_j$ has the same chromatic class as exactly two of these
vertices, it follows that at least one element in
$\{\lambda_{0j},\lambda_{1j},\lambda_{2j},\lambda_{3j}\}$ is $3$, and
at least two elements are $2$. Thus it follows that
  $(\lambda_{0j}+\lambda_{1j}+\lambda_{2j}+\lambda_{3j}-8)\geq 0.$ 

Therefore \eqref{eq:l16} implies that $(\lambda_{03}-1)t_0 +
(\lambda_{03}-1)t_3 \le 0$. Recall that $\lambda_{03}$ is either $1$
or $3$. If $\lambda_{03}=3$, then we have $2t_0 + 2t_3 \le 0$, which
contradicts (Proposition~\ref{pro:tempev}) that $\LL(\Key{D})$ has a
positive integral solution. We conclude that $\lambda_{03}=1$, that
is, $\pi_0$ and $\pi_3$ are adjacent in $\PKey{D}$. This proves \eqref{it:bigtwo}.

We also note that since $\lambda_{03}=1$, \eqref{eq:l16} implies that
\begin{equation}\label{eq:l17}
\sum_{j\in[r],j>3} (\lambda_{0j}+\lambda_{1j}+\lambda_{2j}+\lambda_{3j}-8) t_j=0.
\end{equation}

\noindent

By way of contradiction suppose there is a
vertex $\pi_4$ adjacent to none of
$\pi_0,\pi_1,\pi_2,\pi_3$ in $\PKey{D}$. 
Then each of $\lambda_{04},\lambda_{14},\lambda_{24},\lambda_{34}$ is
at least $2$. 
Using Proposition~\ref{pro:forkeys} and that  $\Key{D}$ is $\{0,4\}$-free, it follows that two of these
$\lambda$s are
$2$, and the other two are $3$. 
Therefore
$(\lambda_{04}+\lambda_{14}+\lambda_{24}+\lambda_{34}-8)=2$. Using
\eqref{eq:l17} we obtain
\begin{equation}\label{eq:l18}
2t_4+\sum_{j\in[r],j>4} (\lambda_{0j}+\lambda_{1j}+\lambda_{2j}+\lambda_{3j}-8) t_j=0.
\end{equation}
We recall that $\lambda_{0j} +\lambda_{1j} + \lambda_{2j} +
\lambda_{3j} - 8 \ge 0$ for every $j > 3$. 
Using this and \eqref{eq:l18}, it follows that $2t_4 \le 0$.
But this contradicts that $\LL(\Key{D})$  \tred{ has a positive integral
solution.}
\end{proof}

\begin{proposition}\label{pro:leaf}
Let $D$ be an optimal drawing of $K_{5,n}$, with $n$
even. Suppose that $\Key{D}$ is $\{0,4\}$-free. Then
$\PKey{D}$ has minimum degree at least $2$.
\end{proposition}

\begin{proof}
By way of contradiction, suppose that $\PKey{D}$ has a vertex of
degree $0$ or $1$.

Suppose first that
$\PKey{D}$ has a vertex of degree $0$. Then the
connectedness of $\PKey{D}$ implies that this is the only vertex in
$\PKey{D}$ (and, consequently, the only vertex in $\Key{D}$). Thus all vertices of $D$ have the same rotation. Since
if $a_i, a_j$ have the same rotation in a drawing $D'$ then $\ucr{D'}{a_i,a_j}=4$, it follows that $\cru(D) \ge 4{n\choose 2} =
2n(n-1)$. Since $Z(5,n) = n(n-2)$ and $D$ is optimal, we must have
$2n(n-1) \le n(n-2)$, but this inequality does not hold for any positive integer
$n$. 

Thus we may assume that $\PKey{D}$ has a vertex of degree $1$. 

Let
$\pi_0,\pi_1,\ldots,\pi_{m-1}$ denote the vertices of $\PKey{D}$.
Without any loss of generality we may assume that $\pi_0$ has degree
$1$ in 
$\PKey{D}$.
For $i,j\in {\tred{[m]}}$, let $\lambda_{ij}$
denote the label of the edge $\pi_i\pi_j$. 

We divide the rest of the proof into two cases.

\vglue 0.3 cm
\noindent{\sc Case 1. }{\em $\PKey{D}$ has a path with $4$ vertices starting at
$\pi_0$.}
\vglue 0.3 cm

Without loss of generality, let 
$\pi_0,\pi_1,\pi_2,\pi_3$ be this path. Since $\pi_0$ is a leaf, it follows that
$\pi_1$ is the only vertex of $\PKey{D}$ adjacent to both $\pi_0$ and
$\pi_2$. We note that then there must be a vertex in $\PKey{D}$ (say $\pi_4$,
without loss of generality) adjacent to
both $\pi_1$ and $\pi_3$, as otherwise it would follow by
Proposition~\ref{pro:horse}(\ref{it:bigtwo}) that $\pi_0$ is adjacent
to $\pi_3$, contradicting that $\pi_0$ is a leaf.  Thus
$(\pi_1,\pi_2,\pi_3,\pi_4,\pi_1)$ is a cycle. 

For $i,j\in \tred{[5]}$, let $\lambda_{ij}$ denote the label of $\pi_i\pi_j$
in $\Key{D}$. 
Since the edges $\pi_0\pi_1,\pi_1\pi_2,\pi_2\pi_3,\pi_3\pi_4$ and
$\pi_1\pi_4$ are all in $\PKey{D}$, it follows that 
$\lambda_{01}=\lambda_{12}=\lambda_{23}=\lambda_{34}=\lambda_{14}=1$.
Now since $\Key{D}$ is $\{0,4\}$-free, using
Proposition~\ref{pro:forkeys} 
it follows that $\lambda_{02} = \tred{\lambda_{04}}
=\lambda_{24}=\lambda_{13}=2$ and (since $\pi_0\pi_3$ is
not in $\PKey{D}$) that 
$\lambda_{03}=3$. 

\vglue 0.2 cm
\noindent{\sc Subcase 1.1. }{\em $\pi_0,\pi_1,\pi_2,\pi_3,\pi_4$ are all the vertices in
$\PKey{D}$.}
\vglue 0.2 cm

In this case the linear system $\LL(\Key{D})$ reads:

\begin{align*}
\begin{matrix}
\begin{tabular}{c c r c r c r c r c r c r}
$E_0$ & : &  $2t_0 $ & $-$ & $ t_1$ & $ $ & $ $ & $+$ & $t_3 $ & $ $ & $ $    $=$ & $0$, \\
$E_1$ & : & $-t_0 $ & $+$ & $ 2t_1$ & $- $ & $t_2 $ & $ $ & $ $ & $-$ & $t_4 $    $=$ & $0$, \\
$E_2$ & : & $ $ & $-$ & $ t_1$ & $+$ & $2t_2 $ & $-$ & $t_3 $ & $$ & $ $    $=$ & $0$, \\
$E_3$ & : & $t_0 $ & $$ & $ $ & $ -$ & $t_2 $ & $+$ & $2t_3 $ & $-$ & $t_4$    $=$ & $0$, \\
$E_4$ & : & $ $ & $-$ & $ t_1$ & $ $ & $ $ & $-$ & $t_3 $ & $+$ & $2t_4 $    $=$ & $0$, 
\end{tabular}
\end{matrix}
\end{align*}

\noindent where for brevity we let $E_i:= E(\pi_i,\Key{D})$ for $i\in [5]$.

\tred{Subtracting $E_4$ from $E_2$, we obtain that $t_2=t_4$. Adding
the equations $E_0,E_1,E_2$, and using $t_2=t_4$, we 
obtain $t_0 = 0$.} Thus the system $\LL(\Key{D})$ has no positive integral solution,
contradicting (by Proposition~\ref{pro:tempev}) the optimality of $D$.

\vglue 0.2 cm
\noindent{\sc Subcase 1.2. }{\em $\PKey{D}$ has a vertex $\pi_5 \notin
  \{\pi_0,\pi_1,\pi_2,\pi_3,\pi_4\}$.}
\vglue 0.2 cm

The connectedness of $\PKey{D}$ implies that $\pi_5$ is
adjacent to $\pi_i$ for some $i\in \{0,1,2,3,4\}$. Since $\pi_0$ is a
leaf only adjacent to \tred{$\pi_1$}, then $i\neq 0$. Since $\pi_1$ already
has degree $3$ in $\PKey{D}$, it follows from
Proposition~\ref{pro:b}\eqref{it:c1} that $i\neq 1$. Thus $i$ is
either $2,3$ or $4$. Since the roles of $2$ and $4$ are symmetric, we may
conclude that $\pi_5$ is adjacent to either $\pi_2$ or to $\pi_3$.

Suppose first that $\pi_5$ is adjacent to $\pi_3$ in $\PKey{D}$.  

In this case $\lambda_{35}=1$. Using Proposition~\ref{pro:forkeys},
that $\Key{D}$ is $\{0,4\}$-free, \tred{that $\pi_0$ is only adjacent to $\pi_1$, and Claim~\ref{cla:ch}},
we obtain 
$\lambda_{05}=\lambda_{25}=\lambda_{45}=2$ and that $\lambda_{15}=3$.
Thus in this case the $0$-th and the $5$-th equations of the system $\LL(\Key{D})$ read:
\begin{align*}
\begin{matrix}
\begin{tabular}{c c r c r c r c r c r c r}
  $E_0$ & : &  $2t_0 $ & $-$ & $ t_1$ &  $+$ & $t_3 $ & $ $ & $ $ & $+$ & $ \sum\limits_{j\in[m],j>5} (\lambda_{0j}-2) t_j$ & $=$ & $0$. \\
$E_5$ & : & $ $ & $+$ & $ t_1$ &  $-$ & $t_3 $ & $+ $ & $2t_5 $ & $+ $ & $ \sum\limits_{j\in[m],j>5} (\lambda_{5j}-2) t_j$ &  $=$ & $0$.
\end{tabular}
\end{matrix}
\end{align*}

\noindent where for brevity we let $E_i:= E(\pi_i,\Key{D})$ for $i=0$
and $5$.

Adding these equations, we get
\begin{equation}\label{eq:la}
2t_0 + 2t_5 + \sum\limits_{j\in[m],j>5} (\lambda_{0j} + \lambda_{5j} - 4   )t_j = 0.
\end{equation}
We now argue that $\lambda_{0j} + \lambda_{5j} - 4\ge0$
whenever $j > 5$. To see this, note that $\pi_0$ and $\pi_5$ are in
the same chromatic class.  If $\pi_j$ is in the same chromatic class,
then, since $\Key{D}$ is $\{0,4\}$-free, it follows that
$\lambda_{0j}$ and $\lambda_{5j}$ are both $2$, and so 
$\lambda_{0j} + \lambda_{5j} - 4\ge0$, as claimed. If
$\pi_j$ is in the other chromatic class, then 
both $\lambda_{0j}$ and $\lambda_{5j}$ are odd. Since
$\pi_0$ is a leaf whose only adjacent vertex is $\pi_1$, it follows
that $\lambda_{0j} = 3$. On the other hand,
$\lambda_{5j}$ is either $1$ or $3$. In particular,
$\lambda_{5j} \ge 1$,
and thus also in this case $\lambda_{0j} + \lambda_{5j}
- 4\ge0$, as claimed. 
It follows from this observation and \eqref{eq:la} that
\begin{equation*}
2t_0 + 2t_5  \le 0,
\end{equation*}

\noindent and so the system $\LL(\Key{D})$ has no positive integral solution,
contradicting Proposition~\ref{pro:tempev}.

Suppose finally that $\pi_5$ is adjacent to $\pi_2$ in
$\PKey{D}$. 

Consider then the path $\pi_0,\pi_1,\pi_2,\pi_5$. Since $\pi_0$ is a
leaf, it follows that $\pi_1$ is the only vertex adjacent to both
$\pi_0$ and $\pi_2$. Now note that $\pi_2$ is the only vertex adjacent 
to both $\pi_1$ and $\pi_5$, since by Proposition~\ref{pro:b}\eqref{it:c1}
$\pi_1$ cannot be incident to any vertex other than $\pi_0,\pi_2$, and
$\pi_4$. Thus Proposition~\ref{pro:horse} applies, and so we must have
that $\pi_0$ and $\pi_5$ are adjacent in $\PKey{D}$. But this is
impossible, since the only vertex in $\PKey{D}$ adjacent to the leaf
$\pi_0$ is 
$\pi_1$.

\vglue 0.3 cm
\noindent{\sc Case 2. }{\em 
$\PKey{D}$ has no path with $4$ vertices starting at
$\pi_0$.
}
\vglue 0.3 cm

We recall that $\pi_0$ is a leaf in $\PKey{D}$. Let $\pi_1$ be the
vertex adjacent to $\pi_0$. 

Suppose first that $\pi_0$ and $\pi_1$ are the only vertices in
$\PKey{D}$. Then $\LL(\Key{D})$ consists of only two equations, namely
$2t_1-t_0=0$ and $2t_0-t_1=0$. This system obviously has no positive
integral solutions, contradicting Proposition~\ref{pro:tempev}.

We may then assume that there is an additional vertex $\pi_2$ in
$\PKey{D}$. By connectedness of $\PKey{D}$, 
and since $\pi_0$ is a leaf, it follows that $\pi_2$
is adjacent to $\pi_1$. 

If $\pi_0,\pi_1,\pi_2$ are the only vertices $\Key{D}$, then the
system $\LL(\Key{D})$ consists of the three equations
$2t_0 - t_1=0$, $-t_0 + 2t_1 -t_2=0$, and
and
$ -t_1+ 2t_2=0$. Adding these equations we obtain $t_0 +
t_2=0$. Thus also in this case $\LL(\Key{D})$  does not have a
positive integral solution, again contradicting
Proposition~\ref{pro:tempev}.

Thus there
must exist an additional vertex $\pi_3$ in $\PKey{D}$.  
Since $\pi_0$
is a leaf, and by assumption (we are working in Case 2) there is no
path with $4$ vertices starting at $\pi_0$, 
it follows that $\pi_3$ must be adjacent to $\pi_1$. 
We already
know that $\lambda_{01}=\lambda_{12}=\lambda_{13}=1$. 
Since $\Key{D}$ is $\{0,4\}$ free, it follows from
Proposition~\ref{pro:forkeys} that
$\lambda_{02}=\lambda_{03}=\lambda_{23}=2$. Thus in this case
$\LL(\Key{D})$ consists of the equations 
 $2t_0 - t_1=0$,
$ -t_0 + 2t_1 -t_2-t_3=0$,
$ -t_1+ 2t_2=0$, and
$-t_1+ 2t_3=0$. 
It is an elementary exercise to show that these equations do not have
a simultaneous positive integral solution, and so in this case we also
obtain a contradiction to Proposition~\ref{pro:tempev}.
\end{proof}


\section{Properties of cores. IV. Girth and maximum size.}\label{sec:someprrk4}

\begin{proposition}\label{pro:e}
Let $D$ be an optimal drawing of $K_{5,n}$, with $n$
even. Suppose that $\Key{D}$ is $\{0,4\}$-free. Then:
\begin{enumerate}
\item\label{it:g2} $\PKey{D}$ has girth $4$. 
\item\label{it:g3} If $v$ is a degree-$2$ vertex in $\PKey{D}$, then $v$ is in a
$4$-cycle in $\PKey{D}$.
\item\label{it:g4} $\PKey{D}$ has at most $7$ vertices.
\end{enumerate}
\end{proposition}

\begin{proof}
By Proposition~\ref{pro:leaf}, the minimum degree of $\PKey{D}$ is at
least $2$. Since $\PKey{D}$ is \tred{simple and bipartite}, it immediately follows that the girth of
$\PKey{D}$ is a positive number greater than or equal to $4$.  Let
$\pi_0,\pi_1,\pi_2,\pi_3$ be a path in $\PKey{D}$. If there is a
vertex other than $\pi_1$ adjacent to both $\pi_0$ and $\pi_2$, or a
vertex other than $\pi_2$ adjacent to both $\pi_1$ and $\pi_3$, then
$\PKey{D}$ clearly has a $4$-cycle, and we are done. Otherwise, 
it follows from Proposition~\ref{pro:horse}\eqref{it:bigtwo} that $\pi_0$ is adjacent
to $\pi_3$, and so 
$(\pi_0,\pi_1,\pi_2,\pi_3,\pi_0)$ is a $4$-cycle. Thus (\ref{it:g2}) follows.

Now let $\pi_1$ be a degree-$2$ vertex in $\PKey{D}$. Since $\PKey{D}$ has
minimum degree at least $2$, using \eqref{it:g2} it obviously follows that there exists a
path $\pi_0, \pi_1, \pi_2, \pi_3$ in $\PKey{D}$. If there is a vertex
adjacent to both $\pi_0$ and $\pi_2$ other than $\pi_1$, then $\pi_1$
is obviously contained in a $4$-cycle. In such a case we are done, so
suppose that this is not the case. Since $\pi_1$ is only adjacent
to $\pi_0$ and $\pi_2$, using that the degree of $\pi_1$ is $2$ it follows
that no vertex other than
$\pi_2$ is adjacent to both $\pi_1$ and $\pi_3$. Thus it follows
from Proposition~\ref{pro:horse}(\ref{it:bigtwo}) that $\pi_0$ and $\pi_3$ are
adjacent in $\PKey{D}$. Thus $\pi_1$ is contained in the $4$-cycle $(\pi_0, \pi_1, \pi_2,
\pi_3, \pi_0)$, and (\ref{it:g3}) follows.

Let $C=(\pi_0, \pi_1, \pi_2, \pi_3,\pi_0)$ be a $4$-cycle in
$\PKey{D}$; the existence of $C$ is guaranteed from
(\ref{it:g2}). By Proposition~\ref{pro:b}\eqref{it:c2} $\PKey{D}$ contains no subgraph isomorphic to
$K_{2,3}$, and so, in $\PKey{D}$, no vertex other than \tred{$\pi_1$ or $\pi_3$} is adjacent to
both $\pi_0$ and $\pi_2$, and no vertex other than \tred{$\pi_2$ or $\pi_0$} is adjacent
to both $\pi_1$ and $\pi_3$. Thus Proposition~\ref{pro:horse}
applies. Using 
Proposition~\ref{pro:b}(\ref{it:c1}) and Proposition~\ref{pro:horse}(\ref{it:bigone}),
we obtain that $\PKey{D}$ has at
most $4$ vertices other than $\pi_0,\pi_1, \pi_2$, and $\pi_3$; that
is, $\PKey{D}$ has at most $8$ vertices in total; moreover, if
$\PKey{D}$ has exactly $8$ vertices, then every vertex of $C$ has
degree $3$. Since  $C$ was an arbitrary $4$-cycle, we have actually
proved that if $\PKey{D}$ has $8$ vertices, then every vertex
contained in a
$4$-cycle must have degree $3$. In view of \eqref{it:g3}, this implies
that if  $\PKey{D}$ has $8$ vertices, then it must be cubic.

Now the unique (up to isomorphism) cubic connected bipartite graph on
$8$ vertices is the $3$-cube. 
Since the $3$-cube contains as a subgraph the graph in
Figure~\ref{fig:k32}, it follows that $\PKey{D}$ cannot have exactly
$8$ vertices.
\end{proof}

\section{The possible cores of an antipodal-free optimal drawing.}\label{sec:thecores}

Our goal in this section is to establish Lemma~\ref{lem:j}, which
states that the core of every
antipodal-free optimal drawing of $K_{5,n}$ is isomorphic to either a
$4$-cycle or to the graph $\Bur$ obtained from the $6$-cycle by adding
an edge joining two diametrically oposed vertices (see Figure~\ref{fig:bur}).

	\begin{figure}[h!]
	\begin{center}
		\scalebox{0.4}{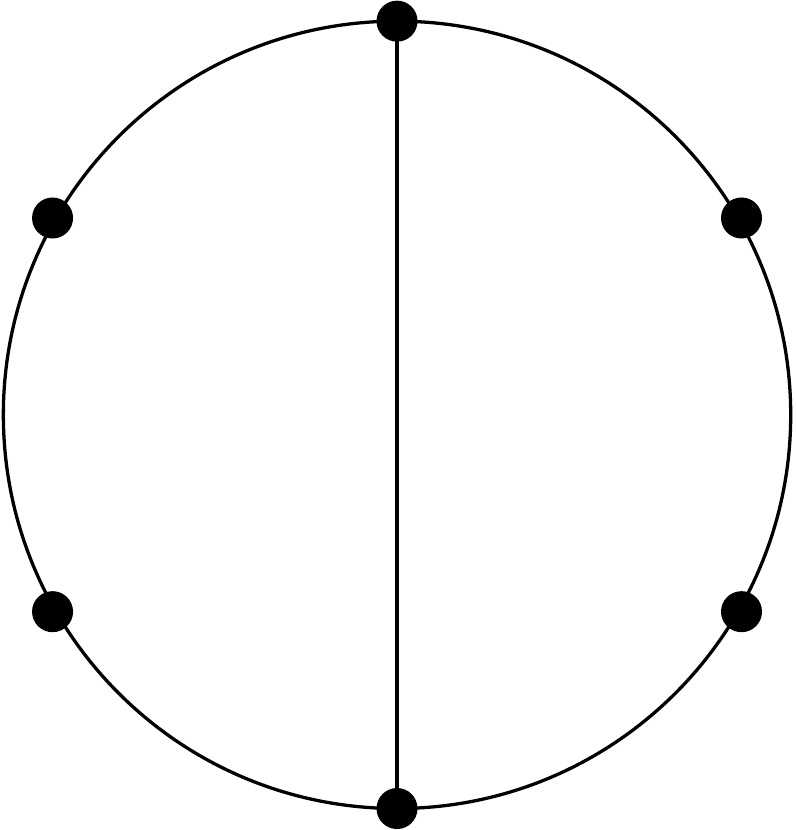}
		 \caption{{\small The graph $\Bur$.}}
		 \label{fig:bur}
	\end{center}
	\end{figure} 

We first show this for the particular case
in which $\Key{D}$ is not only antipodal-free (that is, $0$-free), but also $4$-free:

\begin{proposition}\label{pro:h}
Let $D$ be an optimal drawing of $K_{5,n}$, with $n$
even. If $\Key{D}$ is $\{0,4\}$-free, then $\PKey{D}$ is
isomorphic to the $4$-cycle or to $\Bur$.
\end{proposition}

\begin{proof}
By way of contradiction, suppose that $\PKey{D}$ is isomorphic to
neither a $4$-cycle nor to $\Bur$.
Recall that $\PKey{D}$ has minimum degree at least $2$ (Proposition~\ref{pro:leaf}). We divide the
proof into two cases, depending on whether or not $\PKey{D}$ has
degree-$2$ vertices.

\vglue 0.3 cm
\noindent{\sc Case 1. }{\em $\PKey{D}$ has at least one degree-$2$ vertex.}
\vglue 0.3 cm

By
Proposition~\ref{pro:e}(\ref{it:g4}), $\PKey{D}$ has at most $7$
vertices. 
If all the vertices in $\PKey{D}$ have degree $2$, then (since
$\PKey{D}$ is simple and, by Proposition~\ref{pro:c}\eqref{it:c4},
connected) 
$\PKey{D}$ is a cycle. By Proposition~\ref{pro:e}\eqref{it:g2}, in
this case $\PKey{D}$ is a $4$-cycle, contradicting our assumption at
the beginning of the proof.

Thus we may assume that $\PKey{D}$ has at least one
degree-$3$ vertex. 
Let $H$ be the graph obtained by suppressing the degree-$2$
vertices from $\PKey{D}$. We call the vertices of $\PKey{D}$ that
correspond to the vertices in $H$ (that is, the degree-$3$ vertices of
$\PKey{D}$) the {\em nodes} of $\PKey{D}$. 

It follows from elementary graph theory that $\PKey{D}$ has an even
number of nodes. Since $\PKey{D}$ has at most $7$ vertices, it follows
that $\PKey{D}$ has either $2, 4$, or $6$ nodes. 

\vglue 0.2 cm
\noindent{\sc Subcase 1.1. }{\em $\PKey{D}$ has $6$ nodes.}
\vglue 0.2 cm

Up to isomorphism, there are only two cubic simple graphs on $6$
nodes, namely $K_{3,3}$ and the triangular prism $T_3$ (this is the
simple cubic graph with a matching whose removal leaves two disjoint
$3$-cycles).
 Now $T_3$ has two vertex disjoint $3$-cycles, and so in
order to turn it into a bipartite graph, we must subdivide at least
$2$ edges, that is, add at least two vertices to $T_3$. Since
$\PKey{D}$ has at most $7$ vertices, it follows that $H$ cannot be
isomorphic to $T_3$.

Suppose finally that $H$ is isomorphic to $K_{3,3}$. Since no
bipartite graph on $7$ vertices is a subdivision of $K_{3,3}$, it
follows that $\PKey{D}$ must be itself isomorphic to $K_{3,3}$. Since
$K_{3,3}$ obviously contains $K_{2,3}$ as a subgraph, this contradicts
Proposition~\ref{pro:b}(\ref{it:c2}).

\vglue 0.2 cm
\noindent{\sc Subcase 1.2. }{\em $\PKey{D}$ has $4$ nodes.}  \vglue
0.2 cm In this case $H$ must be isomorphic to $K_4$, the only
  cubic graph on four vertices. It is readily seen that there are only
  two ways to turn $K_4$ into a bipartite graph using at most three
  edge subdivisions. One way is to subdivide once each of the edges
  in a $3$-cycle of $K_4$, and the other way is to subdivide
  (once) two nonadjacent edges (in the latter case, we obtain a graph that has a subgraph
  isomorphic to $K_{2,3}$). By
  Proposition~\ref{pro:b}, neither of these graphs can be
  the core of $D$. \vglue 0.2 cm

\noindent{\sc Subcase 1.3. }{\em $\PKey{D}$ has $2$ nodes.}
\vglue 0.2 cm

In this case $H$ must consist of two vertices joined by three parallel
edges. Since $\PKey{D}$ is bipartite it follows that each of these
edges must be subdivided the same number of times modulo $2$
(subdividing an edge $0$ times being a possibility). Moreover, since
$\PKey{D}$ is simple at least two edges must be subdivided at least
once each.

Now no edge may be subdivided more than twice, as in this case the
result would be a graph with a degree-$2$ vertex belonging to no
$4$-cycle, contradicting Proposition~\ref{pro:e}(\ref{it:g3}).

Suppose now that some edge of $H$ is subdivided exactly twice. Then, since
$\PKey{D}$ has at most $7$ vertices, it follows that two edges of $H$
are subdivided exactly twice, and the other edge of $H$ is not
subdivided. 
Thus it follows that in
this
case $\PKey{D}$ is isomorphic to $\Bur$, contradicting our initial assumption.

Suppose finally that no edge of $H$ is subdivided more than
once. Since $\PKey{D}$ is bipartite, it follows that every edge
of $H$ must be subdivided exactly once. Thus $\PKey{D}$ is isomorphic
to $K_{2,3}$, contradicting Proposition~\ref{pro:b}(\ref{it:c2}).

\vglue 0.3 cm
\noindent{\sc Case 2. }{\em $\PKey{D}$ has no degree-$2$ vertices.}
\vglue 0.3 cm

In this case, $\PKey{D}$ is cubic. By Proposition~\ref{pro:c},
$\PKey{D}$ is bipartite and connected. By
Proposition~\ref{pro:e}(\ref{it:g4}), $\PKey{D}$ has at most $7$
vertices. By elementary graph theory, since $\PKey{D}$ is cubic, then
it has an even number of vertices. Since $\PKey{D}$ is simple, it
follows that $\PKey{D}$ has either $4$ or
$6$ vertices.

Now there are no simple cubic bipartite graphs on $4$ vertices, so
$\PKey{D}$ must have $6$ vertices. 
Up to isomorphism, the only cubic bipartite graph on $6$ vertices is
$K_{3,3}$. But $\PKey{D}$ cannot be isomorphic to $K_{3,3}$, since by 
Proposition~\ref{pro:b}(\ref{it:c2}) $\PKey{D}$ does not contain a
subgraph isomorphic to $K_{2,3}$.
\end{proof}

\begin{proposition}\label{pro:r}
Let $D$ be an antipodal-free, optimal drawing of $K_{5,n}$, with $n$
even. Then $\Key{D}$ is $4$-free. 
\end{proposition}

\begin{proof}
By way of contradiction, suppose that $\Key{D}$ is not $4$-free. Then there exist distinct
rotations $\pi,\pi'$, and white vertices $a_i, a_j$ such that
$\rot{D}{a_i}=\pi$ and $\rot{D}{a_j}=\pi'$, and $\ucr{D}{a_i,a_j}=4$.

{Without loss of generality, suppose that $\ucr{D}{a_i} \le \ucr{D}{a_j}$.} We move, one by one, every vertex $a_j$ with rotation $\pi'$ very
close to $a_i$, so that in the resulting drawing $D'$ we have
$\ucr{D'}{a_j,a_k}=\ucr{D'}{a_i,a_k}$ for every vertex
$k\notin\{i,j\}$.
It is readily checked that the resulting
drawing $D'$ is also optimal, and $\Key{D'}$ has one fewer edge with
label $4$ than $\Key{D}$.  By repeating this process as many times as
needed, we arrive to a drawing $D^o$ such that $\Key{D^o}$ has exactly
one edge with label $4$ (if $\Key{D}$ has exactly one edge with label
$4$ to begin with, then we let $D^o=D$). Denote by $\pi_0, \pi_1$ the vertices of
$\Key{D^o}$ whose joining edge has label $4$. 

If we apply the described process one more time to $D^o$ with $\pi=\pi_0$ and
$\pi'=\pi_1$,  we obtain a $\{0,4\}$-free optimal drawing $E$ of
$K_{5,n}$. By Proposition~\ref{pro:h}, $\PKey{E}$ contains a
$4$-cycle $(\pi_0,\pi_2,\pi_3,\pi_4,\pi_0)$. Now if we apply the process to
$D^o$ with $\pi=\pi_1$ and $\pi'=\pi_0$, then we obtain another
$\{0,4\}$-free optimal drawing $F$ of $K_{5,n}$. 
Note that $\pi_2, \pi_3, \pi_4$ are not affected in the
process, and so $(\pi_1,\pi_2,\pi_3,\pi_4,\pi_1)$ is a $4$-cycle in
$\PKey{F}$. 
{Thus it follows that $\PKey{D^o}$ has
two degree-$3$ vertices $\pi_2$ and $\pi_4$,
plus the vertices $\pi_0,\pi_1,\pi_3$, each of which is joined to both $\pi_2$ and $\pi_4$ with an edge labelled $1$. This contradicts Claim~\ref{cla:ch}.}

\ignore{
For rotations $\pi,\pi'$, we let $\od{\pi,\pi'}$ denote their
{antidistance} (this is, we recall, the size of the shortest antiroute
joining them). Our conclusion about the structure of $\PKey{D^o}$
implies that there exist distinct rotations
$\pi_0,\pi_1,\pi_2,\pi_3,\pi_4$ such that 
(i) there are antiroutes of size $1$
from $\pi_2$ to $\pi_0,\pi_1$, and $\pi_3$; and
(ii) there are antiroutes of size $1$
from $\pi_4$ to $\pi_0, \pi_1$, and $\pi_3$.
}

\end{proof}

\begin{lemma}\label{lem:j}
Let $D$ be an antipodal-free,
optimal drawing of $K_{5,n}$, with $n$ even. Then $\PKey{D}$ is
isomorphic either to the $4$-cycle or to $\Bur$.
\end{lemma}

\begin{proof}
By Proposition~\ref{pro:r}, $\Key{D}$ is $4$-free. By hypothesis
$\Key{D}$ is also $0$-free (since $D$ is antipodal-free), and so $\Key{D}$
is $\{0,4\}$-free. The lemma then follows by Proposition~\ref{pro:h}.
\end{proof}

\section{Proof of Theorem~\ref{thm:main1}.}\label{sec:proofmain}

We need one final result before moving on to the proof of
Theorem~\ref{thm:main1}.
In the following statement and its proof, we sometimes use the notation
$(i,j,k,\ell,m)$ for cyclic permutations (that is, we separate the
elements with commas, as opposed to our usual practice in which for
such a cyclic permutation we would have written $(ijk\ell m)$). 

\begin{proposition}\label{pro:he} 
Let $D$ be a drawing of $K_{5,n}$. Suppose that $\Key{D}$ is $\{0,4\}$-free, and that $\PKey{D}$ is a $4$-cycle $(\pi_0,\pi_1,\pi_2,\pi_3,\pi_0)$.  Suppose that $\pi_0=( 0 1 2 3 4)$. Then there exists an $m\in \{0,1,2,3,4\}$ and a relabelling of $\{0,1,2,3,4\}$ that leaves $\pi_0$ invariant, such that (operations are modulo $5$):
\begin{itemize}
\item $\pi_2=({m}, {m+1}, {m+3}, {m+4},$ $ {m+2})$; and
\item $\{\pi_1,\pi_3\} = \{(m,m+4,m+2,m+3,m+1),(m,{m+4},m+3,m+1,m+2)\}$.
\end{itemize}
\end{proposition}

\begin{proof}
The reverse permutation $\overline{\pi_0}$ of $\pi_0$ is
$( 4 3 2 1 0)$. Since $\pi_0\pi_1$ and $\pi_0\pi_3$ have label
$1$ in $\Key{D}$, it follows that each of $\pi_1$ and $\pi_3$ is
obtained from $\overline{\pi_0}$ by performing one transposition.
Thus there exist distinct $k,m\in\{0,1,2,3,4\}$ such that $\{\pi_1,\pi_3\}=\{(k,k+4,k+2,k+3,k+1),({m} , {m+4}, {m+2}, m+3,m+1)\}$.

Suppose that $k=m+3$. Using a relabelling on $\{0,1,2,3,4\}$ that leaves $(01234)$ invariant, we may assume that $m=2$ and $k=0$. Then $\{\pi_1,\pi_3\} = \{(04231),(03214)\}$. Now since the edge joining $\pi_2$ to each of $\pi_1$ and $\pi_3$ in $\Key{D}$ has label $1$, it follows that there are antiroutes of size $1$ from $\pi_2$ to each of $\pi_1$ and $\pi_3$. It is easy to check that the only such possibility is that $\pi_2=(04132)$. Using the relabelling $j\mapsto j-2$ on $\{0,1,2,3,4\}$, we get $\{\pi_0,\pi_1,\pi_2,\pi_3\}=\{(01234),(01432),(03241),(04231)\}$. But then $\Key{D}$ is the labelled graph in Fig.~\ref{fig:chim}, contradicting Proposition~\ref{pro:chimuelo}.  An analogous contradiction is obtained under the assumption $k=m+2$. Thus $k=m+1$ or $k=m+4$.

Suppose that $k=m+1$. Thus $\{\pi_1,\pi_3\}=\{(m+1,m,m+3,m+4,m+2),(m,m+4,m+2,m+3,m+1)\}$. Using the relabelling $j\mapsto j-1$ on $\{0,1,2,3,4\}$ (which obviously leaves $(01234)$ invariant), we obtain $\{\pi_1,\pi_3\}=\{(m,m+4,m+2,m+3,m+1),(m+4,m+3,m+1,m+2,m)\}=\{(m,m+4,m+2,m+3,m+1),(m,m+4,m+3,m+1,m+2)\}$, as required. Finally, since the edge joining $\pi_2$ to each of $\pi_1$ and $\pi_3$ in $\Key{D}$ has label $1$, it follows that $\pi_2=(m,m+1,m+3,m+4,m+2)$. The case $k=m+4$ is handled in a totally analogous manner.
\end{proof}

\begin{proposition}\label{pro:wasc}
Suppose that $D$ is a drawing of $K_{5,n}$. Suppose that $\Key{D}$ is
$\{0,4\}$-free, and that $\PKey{D}$ is isomorphic to $\Bur$. Let the
vertices of $\PKey{D}$ be labeled
$\pi_0,\pi_1,\pi_2,\pi_3,\pi_4,\pi_5$, so that
$(\pi_0,\pi_1,\pi_2,\pi_3,\pi_0)$ and $(\pi_0,\pi_4,\pi_5,\pi_3,\pi_0)$ are $4$-cycles.
Suppose that $\pi_0=( 0 1 2 3 4)$. Then there exists
an $m\in \{0,1,2,3,$ $4\}$ and a relabelling of $\{0,1,2,3,4\}$ that leaves $\pi_0$ invariant, such that (operations are modulo $5$):
\begin{itemize}
\item $\pi_3=({m}, {m+4}, {m+3},$ $ {m+1},m+2)$;
\item $\{(\pi_1,\pi_2),(\pi_4,\pi_5)\}=
\{((m,m+4,m+2,m+3,m+1), (m,m+1,m+3,m+4,m+2)),((m,m+1,m+4,m+3,m+2),(m,m+2,m+3,m+1,m+4))\}$.
\end{itemize}
\end{proposition}

\begin{proof}
By Proposition~\ref{pro:he}, there exists an $m\in\{0,1,2,3,4\}$ such
that $\pi_2=(m,m+1,m+3,m+4,m+2)$ and
$\{\pi_1,\pi_3\}=A:=\{(m,m+4,m+2,m+3,m+1),(m,m+4,m+3,m+1,m+2)\}$. By the
same proposition, there exists a $k\in\{0,1,2,3,4\}$ such that
$\pi_5=(k,k+1,k+3,k+4,k+2)$ and
$\{\pi_3,\pi_4\}=B:=\{(k,k+4,k+2,k+3,k+1),(k,k+4,k+3,k+1,k+2)\}$.

Since $\pi_2\neq\pi_5$, it follows that $m\neq k$. Thus $k$ is either $m+1,m+2,m+3$, or $m+4$. Note that if $k=m+2$ or $k=m+3$ then $A\cap B=\emptyset$, which contradicts that $\{\pi_3\}=A\cap B$. Thus $k$ is either $m+1$ or $m+4$. 

We work out the details for the case $k=m+1$; the case $k=m+4$ is handled in a totally analogous manner. Since $\{\pi_3\}=A\cap B$, it follows that
$\pi_3=(m,m+4,m+2,m+3,m+1)=(m+1,m,m+4,m+2,m+3)$. Therefore
$\pi_1=(m,m+4,m+3,m+1,m+2)=(m+1,m+2,m,m+4,m+3)$,
$\pi_2=(m,m+1,m+3,m+4,m+2)=(m+1,m+3,m+4,m+2,m)$,
$\pi_4=(m+1,m,m+3,m+4,m+2)$, and
$\pi_5=(m+1,m+2,m+4,m,m+3)$. Using the relabelling $j\to j-1$ on $\{0,1,2,3,4\}$ (which leaves $(01234)$ invariant), we obtain $\pi_1=(m,m+1,m+4,m+3,m+2)$, $\pi_2=(m,m+2,m+3,m+1,m+4)$, $\pi_3=(m,m+4,m+3,m+1,m+2)$ $\pi_4=(m,m+4,m+2,m+3,m+1)$, and $\pi_5=(m,m+1,m+3,m+4,m+2)$. 
\end{proof}

\begin{proof}[Proof of Theorem~\ref{thm:main1}]
Let $D$ be an antipodal-free drawing of $K_{5,n}$, with $n$ even. In
view of Proposition~\ref{pro:same} (see Remark~\ref{rem:clean}), we
may assume that $D$ is clean, so that $\Key{D}$ and $\PKey{D}$ are
well-defined.

In view of Lemma~\ref{lem:j}, $\PKey{D}$ is isomorphic either to the $4$-cycle or to $\Bur$.

\smallskip
\noindent{\sc Case 1.} {\em $\Key{D}$ is isomorphic to $\Bur$.}
\smallskip

In this case $\Key{D}$ has $6$ vertices, which we label $\pi_0, \pi_1,
\pi_2, \pi_3, \pi_4, \pi_5$, so that $(\pi_0,\pi_1,\pi_2,\pi_3,\pi_0)$
and $(\pi_0,\pi_4,\pi_5,\pi_3,\pi_0)$ are $4$-cycles. For
$i,j\in\{0,1,2,3,4,5\}$, $i\neq j$, let $\lambda_{ij}$ be the label of
the edge $\pi_i\pi_j$. Since $(\pi_0,\pi_1,\pi_2,$ $\pi_3,\pi_0)$ and $(\pi_0,\pi_4,\pi_5,\pi_3,\pi_0)$ are $4$-cycles in $\PKey{D}$, it follows that all the edges in these $4$-cycles have label $1$ in $\Key{D}$; that is, $\lambda_{01}=\lambda_{12} = \lambda_{23} = \lambda_{03} =\lambda_{04}=\lambda_{45}=\lambda_{35}=1$. 
By Proposition~\ref{pro:forkeys},
$\lambda_{02}$ is even. Since $\Key{D}$ is
antipodal-free, and (by Property (\ref{pro:pfu}) of a clean drawing)
$\lambda_{ij}\le 4$ for all $i,j$, it follows that $\lambda_{02}$ is either $2$ or $4$. By Proposition~\ref{pro:r} $\Key{D}$ is $4$-free, hence $\lambda_{02}=2$. The same argument shows that 
$\lambda_{05}=\lambda_{13}=\lambda_{14}=\lambda_{25}=\lambda_{34}=2$. Since $\lambda_{35}=1$ and $\lambda_{13}=2$, by Proposition~\ref{pro:forkeys}, $\lambda_{15}$ is odd. 
If $\lambda_{15}=1$, then $\{\pi_0, \pi_5\}\cup \{\pi_1,\pi_2,\pi_4\}$ is a $K_{2,3}$ in $\PKey{D}$, contradicting Proposition~\ref{pro:forkeys}; thus $\lambda_{15}=3$. An analogous argument shows that $\lambda_{24}=3$.

The linear system $\LL(\Key{D})$ associated to $\Key{D}$ (see
Definition~\ref{def:linsys}) is then:
\begin{align}
\begin{matrix}\label{sys:bur}
\begin{tabular}{c c r c r c r c r c r c r c r c r}
  $E_0$ & : &  $2t_0 $ & $-$ & $ t_1$ &  &  & $-$  & $t_3 $ & $-$ & $t_4$ & & &  $=$ & $0$. \\
  $E_1$ & : &  $-t_0 $ & $+$ & $2 t_1$& $-$  & $t_2$ &  &  & &  & $+$& $t_5$ & $=$ & $0$. \\
  $E_2$ & : &  & $-$ & $t_1$ & $+$ & $2t_2$ & $-$ & $t_3$ & $+$ & $t_4$ & & & $=$ & $0$. \\
  $E_3$ & : &  $-t_0$ &  &  & $-$ & $t_2$ & $+$ & $2t_3$ &  &  & $-$ & $t_5$ & $=$ & $0$.\\
  $E_4$ & : & $-t_0$ &  &  & $+$  & $t_2$ &  &  & $+$ & $2t_4$ & $-$ & $t_5$ & $=$ & $0$. \\
  $E_5$ & : &  & $+$ & $t_1$ &  &  & $-$ & $t_3 $ & $-$ & $t_4$ & $+$ & $2t_5$ & $=$ & $0$.  
\end{tabular}
\end{matrix}
\end{align}

It is straightforward to check that if $(t_0,t_1,t_2,t_3,t_4, t_5)$ is a positive solution to this system, then $t_1=t_2$, $t_4=t_5$ and $t_0 = t_3 = t_1+t_4$. By Proposition~\ref{pro:tempev}, this implies that $n\equiv 0$ $(\mdl{4})$. This proves \eqref{it:ma1}.

We have thus proved that the white vertices of $D$ are partitioned
into $6$ classes $\cc_0,\cc_1,\cc_2,\cc_3, \cc_4, \cc_5$, such that
$|\cc_1|=|\cc_2|$, $|\cc_4|=|\cc_5|$,
$|\cc_0|=|\cc_3|=|\cc_1|+|\cc_4|$, and such that for $i=0,1,2,3,4,5$,
each vertex in $\cc_i$ has rotation $\pi_i$. Let $r:=|\cc_1|$ and
$s:=|\cc_4|$, so that $|\cc_2|=r$, $|\cc_5|=s$, and
$|\cc_0|=|\cc_3|=r+s$.  Note that $4(r+s)=n$.

If necessary, relabel $\{0,1,2,3,4\}$ so that $\pi_0=( 0 1 2 3 4)$. By Proposition~\ref{pro:wasc}, perhaps after a further relabelling of $\{0,1,2,3,4\}$ (that leaves $\pi_0$ invariant), there exists an $m\in \{0,1,2,3,$ $4\}$ such that $\pi_3=({m}, {m+4}, {m+3},$ $ {m+1},m+2)$, and $\{(\pi_1,\pi_2),(\pi_4,\pi_5)\}=\{((m,m+4,m+2,m+3,m+1), (m,m+1,m+3,m+4,m+2)),((m,m+1,m+4,m+3,m+2),(m,m+2,m+3,m+1,m+4))\}$. Now perform the further relabelling $j\mapsto j-m$. After this relabelling (which again leaves $\pi_0$ invariant), we have $\pi_3=(04312)$ and $\{(\pi_1,\pi_2),(\pi_4,\pi_5)\}=\{((04231),(01342)),((01432),(02314))\}$. 

We have thus proved that (perhaps after a relabelling of $\{0,1,2,3,4\}$) there exist integers $r,s$ such that $D$ has $r+s$ vertices with rotation $\pi_0=(01234)$, $r$ vertices with rotation $\pi_1=(04231)$, $r$ vertices with rotation $\pi_2=(01342)$, $r+s$ vertices with rotation $\pi_3=(04312)$, $s$ vertices with rotation $\pi_4=(01432)$, and $s$ vertices with rotation $\pi_5=(02314)$. That is, $D$ is isomorphic to the drawing $D_{r,s}$ from Section~\ref{sec:spdr}.

\medskip
\noindent{\sc Case 2.} {\em $\Key{D}$ is isomorphic to the $4$-cycle.}
\medskip

In this case $\Key{D}$ has $4$ vertices, which we label $\rho_0,\rho_1,\rho_2,\rho_3$, so that $(\rho_0,\rho_1,\rho_2,\rho_3,\rho_0)$ is a cycle. The linear system $\LL(\Key{D})$ associated to $\Key{D}$ is the one that results by taking $t_4=t_ 5=0$ in the linear system \eqref{sys:bur}, and omitting the equations $E_4$ and $E_5$.

It is straightforward to check that if $(t_0,t_1,t_2,t_3)$ is a solution to this system, then $t_0=t_1=t_2=t_3$. By Proposition~\ref{pro:tempev}, this implies that $n\equiv 0$ (mod $4$). This proves \eqref{it:ma1}.

Thus the white vertices of $D$ are partitioned into $4$ classes $\cc_0,\cc_1,\cc_2,\cc_3$, each of size $n/4$, so that each vertex in class $\cc_i$ has rotation $\rho_i$.

Label the vertices $0,1,2,3,4$ so that $\rho_0=( 0 1 2 3 4)$. Then, by Proposition~\ref{pro:he}, possibly after a relabelling of $\{0,1,2,3,4\}$ that leaves $\rho_0$ invariant, there is an $m\in\{0,1,2,3,4\}$ such that $\rho_2=({m},  {m+1},  {m+3},  {m+4},$ $ {m+2})$, and $\{\rho_1,\rho_3\} = \{(m,m+4,m+2,m+3,m+1),( m, {m+4},m+3,m+1,m+2)\}$. 
Now we perform the relabelling $j\mapsto j-m$ on $\{0,1,2,3,4\}$ (which obviously leaves $\rho_0$ invariant), we obtain $\rho_2=(01342)$ and $\{\rho_1,\rho_3\}=\{(04231),(04312)\}$. 

We have thus proved that $D$ has $r$ vertices with rotation $(01234)$, $r$ vertices with rotation $(01342)$, $r$ vertices with rotation $(04231)$, and $r$ vertices with rotation $(04312)$. That is, $D$ is isomorphic to the drawing $D_{r,0}$ from Section~\ref{sec:spdr}, with $r=n/4$. 
\end{proof}



\ignore{
\section{Concluding remarks.}\label{sec:concludingremarks}
Can we put really something here? It's already a very long
paper. Perhaps a discussion of what makes it hard for $n$ odd?
}


\end{document}